\documentclass[preprint]{siamart171218}

\usepackage[utf8]{inputenc}
\usepackage{amssymb}
\usepackage{amsmath}

\usepackage{algpseudocode,algorithm,algorithmicx}

\usepackage{bm}
\usepackage{pdflscape}
\usepackage{afterpage}
\usepackage{rotating}

\usepackage{multirow}
\usepackage{color}

\newtheorem{remark}{Remark}

\usepackage{geometry}
\usepackage{subcaption}

\usepackage{mathtools}

\usepackage{placeins}

\let\Oldsection\section
\renewcommand{\section}{\FloatBarrier\Oldsection}

\let\Oldsubsection\subsection
\renewcommand{\subsection}{\FloatBarrier\Oldsubsection}

\let\Oldsubsubsection\subsubsection
\renewcommand{\subsubsection}{\FloatBarrier\Oldsubsubsection}

\definecolor{myRed}{rgb}{1.0, 0.0, 0.0}

\definecolor{myBlue}{rgb}{0.2, 0.2, 1.0}

\definecolor{myPurp}{rgb}{0.5, 0.0, 1.0}

\newcommand{\defeq}{\vcentcolon=}

\newcommand*\Let[2]{\State #1 $\gets$ #2}
\newcommand*\Define[2]{\State \textbf{Define } #1 $\defeq$ #2}

\renewcommand{\chi}{\mathcal{X}}

\newcommand{\He}{\mathcal{H}}

\newcommand{\gG}{{\bm G}}

\newcommand{\gM}{{\bm M}}
\newcommand{\gD}{{\bm D}}

\newcommand{\fV}{\mathcal{V}}

\newcommand{\bigO}{\mathcal{O}}
\newcommand{\incr}[1]{ { \mathit{h}_#1  } }
\newcommand{\vu}{{\bm u}}
\newcommand{\vv}{{\bm v}}
\newcommand{\vx}{{\bm x}}
\newcommand{\vy}{{\bm y}}
\newcommand{\vp}{{\bm p}}
\newcommand{\vP}{{\bm P}}
\newcommand{\vQ}{{\bm Q}}

\newcommand{\vg}{{\bm \gamma}}
\newcommand{\vn}{{\bm n}}

\newcommand{\vX}{{\bm X}}
\newcommand{\vhX}{{\bm \chi}}
\newcommand{\vPhi}{{\bm \Phi}}
\newcommand{\vphi}{\mathit{\bm \Phi }}

\newcommand{\vT}{{\bm T}}
\newcommand{\vI}{{\bm I}}

\newcommand{\fundI}{I}

\newcommand{\R}{{\mathbb{R} }}

\renewcommand{\div}{\nabla \cdot}
\newcommand{\Div}{\,\text{div}}
\newcommand{\grad}{\nabla}
\newcommand{\Laplace}{\Delta}

\newcommand{\tr}{\text{tr}}
\newcommand{\adj}{\text{adj}}

\renewcommand{\d}{\; d}

\title{A Diffusion-Driven Characteristic Mapping Method for Particle Management}

\author{Xi-Yuan Yin\thanks{The Department of Mathematics and Statistics, McGill University, Montreal, Canada  H3A 0B9.}
\and Linan Chen\footnotemark[1]
\and Jean-Christophe Nave\footnotemark[1]  \thanks{Corresponding author. \newline \hspace*{13pt} E-mail addresses: \email{xi.yin@mail.mcgill.ca} (XYY). \email{lnchen@math.mcgill.ca} (LNC). \email{jcnave@math.mcgill.ca} (JCN).}}

\begin{document}

\maketitle

\begin{abstract}
We present a novel particle management method using the Characteristic Mapping framework. In the context of explicit evolution of parametrized curves and surfaces, the surface distribution of marker points created from sampling the parametric space is controlled by the area element of the parametrization function. As the surface evolves, the area element becomes uneven and the sampling, suboptimal. In this method we maintain the quality of the sampling by pre-composition of the parametrization with a deformation map of the parametric space. This deformation is generated by the velocity field associated to the diffusion process on the space of probability distributions and induces a uniform redistribution of the marker points. We also exploit the semigroup property of the heat equation to generate a submap decomposition of the deformation map which provides an efficient way of maintaining evenly distributed marker points on curves and surfaces undergoing extensive deformations.
\end{abstract}


%% main text
\section{Introduction}

The parametrization of a curve or surface has many applications in computer graphics, computational geometry and geometric modelling. In scientific computing, where physics modelling and simulation require solving surface PDEs, a parametric representation is often useful for simplifying the surface equations or for generating computational meshes. This parametrization is often time-dependent such as in the case of fluid interfaces in multiphase flows, where the extensive deformation of the surfaces can deteriorate their numerical accuracy. In this context, the numerical resolution of the surfaces are of special importance since the simulation of interfacial dynamics often relies on solving stiff, high order PDEs such as the Cahn-Hilliard equations \cite{gera2017cahn, gera2018modeling}, where proper spatial resolution is crucial for capturing topological transitions and for improving computational efficiency. In order to maintain a good representation of a parametric surface, current research largely focus on two main desirable properties for the parametrization: angle-preservation and area-preservation. 

An angle-preserving, or conformal parametrization guarantees that the pulled back metric on the parametric space differs from the flat metric only by a scalar multiplicative factor. There has been extensive research in the field of conformal maps and application to surface parametrization. Conformal parametrization of genus zero surfaces and its application to patching more complicated surfaces has been studied in \cite{gu2004genus,gu2003global}. For a given surface, conformal parametrizations are not unique and are generally not area-preserving. The area stretching will depend on the chosen conformal parametrization and on the intrinsic curvature of the surface; it may even grow exponentially in protruding regions of the surface. This motivated \cite{jin2004optimal} to design a method which generates a global conformal parametrization minimizing some chosen energy functional of the area element.

For curved surfaces, it is generally not possible to find a parametrization which is both conformal and area-preserving. For certain applications such as surface sampling, an equiareal parametrization is preferred as the density of the sample points will be uniform over the surface. For many Lagrangian particle methods used in interface tracking, appropriate particle redistribution and reinitialization methods are necessary to maintain the accuracy of the surface representation and to prevent artificial topological changes \cite{hieber2005lagrangian, bergdorf2006lagrangian}. Methods for computing equiareal parametrizations include \cite{balmelli2002space, yoshizawa2004fast}, which propose a relaxation algorithm based on some stretch factor computed from mesh point distances. Winslow's rezoning algorithm \cite{winslow1981adaptive} solves a system of variable diffusion equations on the coordinate functions in order to prescribe the Jacobian determinant. Many of these approaches can also be expressed in the framework of adaptive moving mesh methods where a moving-mesh PDE is used to evolved a numerical mesh in order to better resolve the different scales that a solution to an evolution equation may exhibit \cite{huang2010adaptive}. For instance, in \cite{bergdorf2005multilevel}, a moving-mesh PDE was employed as a redistribution map to reinitialize sample point positions in order to resolve the multiscale solutions generated from convection dominated flows.

%The problem of conformal or equiareal parametrizations is also closely related to the sphere packing and Voronoi tessellation problems. Mesh generation methods which arise from this approach typically work directly with the positioning of the sample points on the surface and with the shapes of the polyhedrons they generate \cite{turk1991generating, turk1992re}. On flat spaces, centroidal Voronoi tessellations and Delaunay triangulations can be used to improve the distribution of sample points \cite{lloyd1982least, du1999centroidal}. For curves spaces, direct application of these methods proves to be challenging due to the need of computing geodesic distances. Examples of algorithms for computing geodesic Voronoi diagrams include \cite{mitchell1987discrete, kunze1997geodesic, kimmel2000fast}.

We propose in this paper a novel framework for generating equiareal parametrizations of time-dependent curves or surfaces using the Characteristic Mapping method (CM): a numerical framework for the advection equation based on the gradient-augmented level-set method \cite{nave2010gradient} and the reference map technique \cite{kamrin2012reference}. The CM method was used in \cite{CM, CME} to solve the linear transport and the 2D incompressible Euler equations; it consists in computing the backward-in-time deformation map generated by an incompressible velocity field. Quantities transported under this velocity can be directly obtained as the pullback of the initial condition by this map. In this paper, we use the CM method on a compressible velocity field in order to transport a density field evolving under its continuity equation. The evolving density distribution can be obtained as pullback by the map, i.e. as a volume form. For equiareal parametrization, we define a probability distribution on the parametric space corresponding to a constant scaling of the area element associated with the parametrization. This density is interpreted as the initial condition of a diffusion equation which we write in conservation form. The velocity field associated with this conservation law generates a characteristic map which transports the initial density towards a uniform density. This redistribution map is pre-composed with the parametrization function to generate an equiareal parametrization. Note that since the image space of the parametrization is unchanged, this redistribution does not affect the accuracy of the surface location but improves its numerical representation. 

The idea of diffusion flows has been used extensively outside the context of surface parametrization. For fluid simulations, the diffusion velocity method, as a generalization to particle methods, has seen many applications in numerical simulations of transport-dispersion and of viscous flows, for instance in \cite{lacombe1999presentation, mycek2016formulation}. Further extensive theoretical and numerical analysis for the blob-particle method for linear and non-linear diffusion can also be found in \cite{carrillo2019blob}. It is well-known that the diffusion equation can be viewed as the $L^2$ gradient descent of the Dirichlet energy functional, this allows for fast and robust methods that approximate an uniform redistribution of the density. In the context of optimal transport, the diffusion equation can also be seen as a gradient descent of the Gibbs–Boltzmann entropy under the Wasserstein-2 metric as analysed in \cite{jordan1998variational}. There has been extensive research in the field of optimal transport specifically concerning gradient descent and geodesic flows \cite{benamou2016augmented, peyre2015entropic, carlier2017convergence, carrillo2019primal, benamou2016augmented}. Notably, the fluid mechanics interpretation of the Monge-Kantorovich problem in \cite{benamou2000computational, benamou2002monge} is most closely related to the CM framework presented in this paper. An overview of the main concepts in optimal transport and gradient flows in the space of probability densities can be found in \cite{santambrogio2017euclidean} and more complete surveys in \cite{ambrosio2008gradient, santambrogio2015optimal, villani2008optimal}. In the context of generating equiareal parametrization, the optimal transport methods have also been investigated, for instance in \cite{zhao2013area, su2016area, su2014optimal}.

There are three main features of the method presented in this paper. Firstly, we propose a novel framework for the application of density transport to the problem of maintaining an equiareal parametrization of a moving surface: the area element of the parametrization is interpreted as a probability density whose change in time is penalized by an $L^2$ gradient descent on its Dirichlet energy. Secondly, the reparametrization method proposed here acts purely on the parametric space, the new parametrization function is obtained as a pre-composition of a transport map with the original parametrization. Consequently, the reparametrization does not affect the precision of the surface and its movements but improves its numerical resolution. Thirdly, the method provides arbitrary resolution of the parametrization. This is possible on one hand due to the functional definition of the transport maps in the parametric and the ambient spaces provided by the semi-Lagrangian approach of the CM method. Indeed, once the redistribution map is computed, resampling of the surface can be readily done by map evaluations; no additional computations are needed. On the other hand, the CM method uses the semigroup structures of the deformation maps to leverage the separation of scales: long time deformations can be decomposed into short time submaps which can be computed on coarser grids \cite{CME}. Since the parametrized surfaces can undergo arbitrary deformations of various scales, the method is effective in achieving faster computational times while maintaining arbitrary resolution of the parametrization.

The paper is organized as follows: in section 2, we present the mathematical framework of the characteristic mapping method for the transport of densities as well as the diffusion flow used for the area redistribution. Section 3 provides an overview of the numerical implementation of the method with some error estimates. We present in section 4 the application of the redistribution method to the maintenance of an equiareal parametrization for moving surfaces. This section also contains the CM framework for the evolution of surfaces in a 3D ambient flow. This is the explicit parametric counterpart of the level-set method used in \cite{CM}. Section 5 presents some numerical results for the application of this method to the evolution of parametric curves and surfaces. Finally, section 6 contains some concluding remarks and potential directions for future research.

\section{Mathematical Formulation}
\subsection{Characteristic Mapping Method for Density Transport} \label{subsec:CM}
The use of the characteristic mapping method for linear advection and self-advection in the case of the Euler equations have been studied in \cite{CM, CME}. In this section, we extend the CM framework to the transport of density distributions.

For a given velocity field $\vu$ defined on some domain $U \subset \R^d$, characteristic curves of the velocity are given by the solution of the IVP
\begin{subequations} \label{eqGroup:charCurve}
\begin{align} 
\frac{d}{dt} \vg(t) = \vu(\vg(t), t)  & \quad \quad t \in \R_+ , \\
\vg(0) = \vg_0 & \quad \quad \vg_0 \in U . 
\end{align}
\end{subequations}

The backward characteristic map $\vX$ is the backward solution operator of the characteristic ODE in the sense that
\begin{gather}
\vX(\vg(t), t) = \vg_0,
\end{gather}
for all initial conditions $\vg_0$ and all $t \in \R_+$.

We can check that $\vX$ satisfies the PDE
\begin{subequations} \label{eqGroup:charMap}
\begin{align} 
\partial_t \vX + (\vu \cdot \grad) \vX = 0 & \quad \quad  \forall (\vx, t) \in U \times \R_+ ,\\
 \vX(\vx, 0) = \vx .&
\end{align}
\end{subequations}

The characteristic map possesses a semigroup structure which allows for the decomposition of a long-time map into the composition of several submaps. We denote by $\vX_{[\tau_{i+1}, \tau_i]}$ the backward characteristic map for the time-interval $[\tau_i, \tau_{i+1}]$. This means that for any characteristic curve $\vg$ satisfying \eqref{eqGroup:charCurve}, we have that 
\begin{gather}
\vX_{[\tau_{i+1}, \tau_i]}(\vg(\tau_{i+1})) = \vg(\tau_i) .
\end{gather}

\begin{remark}
Note that the notation for the characteristic maps used here is slightly different from the one used in \cite{CM}. The backward map in the interval $[\tau_i, \tau_{i+1}]$ is denoted $\vX_{[\tau_{i+1}, \tau_i]}$ with the endpoints of the interval inverted to highlight that it is the backward in time transformation.
\end{remark}

The global time characteristic map can then be split into submaps using the following decomposition
\begin{gather} \label{eq:SubmapDecomp}
\vX (\cdot, t) = \vX_{[t, 0]} =  \vX_{[ \tau_1, 0]} \circ  \vX_{[\tau_2, \tau_1]} \circ \cdots \circ  \vX_{[\tau_{m-1}, \tau_{m-2}]} \circ  \vX_{[t, \tau_{m-1}]} .
\end{gather}
We will denote the global backward characteristic map $\vX_{[t, 0]}(\vx)$ by $\vX_B(\vx, t)$.

For smooth, divergence-free velocity fields, these characteristic maps are known to be diffeomorphisms for all times. Further details on this method and its application to the advection equation can be found in \cite{CME}.

Without the divergence-free assumption on the velocity field, the diffeomorphism property is not guaranteed. The evolution of the Jacobian determinant of the characteristic maps can be computed along characteristic curves using Jacobi's formula:
\begin{gather}
\frac{d}{dt} \det \left(\grad \vX_B (\vg(t),t) \right) = \tr \left(  \adj(\grad \vX_B)  \frac{d}{dt} \grad \vX_B (\vg(t),t) \right) \\ \nonumber
 =   \tr \left(  - \adj (\grad \vX_B) \cdot \grad \vX_B \cdot \grad \vu(\vg(t),t) \right) = - \det (\grad \vX_B) \div \vu ,
\end{gather}
where $\adj$ denotes the adjugate matrix. That is
\begin{gather} \label{eq:logdet}
\frac{d}{dt} \log \det ( \grad \vX_B) = (\partial_t + \vu \cdot \grad ) \log \det ( \grad \vX_B) = - \div \vu ,
\end{gather}
and the absence of finite-time blow-up of the time-integral of the divergence on all characteristic curves is required for the characteristic maps to remain diffeomorphisms.

Under the assumption that the velocity $\vu$ generates diffeomorphic characteristic maps, we can extend the CM method to the density transport problem. Consider the continuity equation
\begin{subequations} \label{eqGroup:contEqn}
\begin{align} 
\partial_t \rho + \div (\rho \vu) = 0 & \quad \quad  \forall (\vx, t) \in U \times \R_+ ,\\
 \rho(\vx, 0) = \rho_0(\vx) .&
\end{align}
\end{subequations}

In this paper, for the purpose of particle management, we will restrict our attention to the case of positive densities $\rho$ bounded away from 0. That is, we assume there exists some $a > 0$ such that $\rho(\vx, t) \geq a$ $\forall (\vx, t) \in U \times \R_+$. We also assume that $\rho_0$ is a probability density, in particular, \eqref{eqGroup:contEqn} implies that $\rho$ will then remain a probability density for all $t$.

The solution $\rho(\vx, t)$ can then be obtained from the characteristic map as follows:
\begin{gather}
\rho(\vx, t) = \rho_0 (\vX_B(\vx, t)) \det \grad \vX_B .
\end{gather}
The submap decomposition \eqref{eq:SubmapDecomp} can also be applied to this pullback.

\subsection{Density Redistribution}  \label{subsec:densityRedist}
In this section we apply the deformation map $\vX_B$ in the context of measure transport. Let $m$ denote the uniform probability measure on $U$, that is, $m(U) = \nolinebreak 1$ with constant density with respect to the Lebesgue measure. Let $\mu$ be the initial probability measure continuous with respect to uniform with density $\rho_0 \d m$. We define $\mu_t$ as the pullback measure of $\mu$ by $\vX_B(\cdot, t)$. That is, $\vX_B$ is a deterministic coupling between the probability spaces $(U, \mathcal{B}, \mu)$ and $(U, \mathcal{B}, \mu_t)$. From \eqref{eqGroup:contEqn}, we have that $\mu_t$ has density $\d \mu_t = \rho_t \d m$.

The characteristic map is generated from a chosen velocity field $\vu$. For instance, one can use the Moser flow \cite{dacorogna1990partial} given by
\begin{gather}
\vu(\vx, t) = \frac{\grad \Laplace^{-1}(\rho_0 - \rho_1)}{(1-t)\rho_0 + t \rho_1} ,
\end{gather}
to generate a coupling between two strictly positive probability densities $\rho_0$ and $\rho_1$ such that $\vX_{[1,0]}$ moves the base density $\rho_1$ to a target $\rho_0$.

For the method presented in this paper, $\vX_B$ will be a coupling between the target density $\rho_0$ and the uniform density. As a result, pullback by $\vX_B$ will ``uniformize'' $\rho_0$. For parametrized curves and surfaces, $\rho_0$ will correspond to the arclength or area-element functions defined on $U$. Pre-composition of the parametrization with $\vX_B$ will then yield an equiareal parametrization. We will examine this in further details in section \ref{sec:surfAdv}.

The transport map is generated from a heat equation for the densities: $\partial_t \rho = \Laplace \rho$ (with Neumann or periodic boundary condition), which we write as a continuity equation:
\begin{subequations} \label{eqGroup:contHeatEqn}
\begin{align} 
\partial_t \rho + \div (- \grad \rho ) = \partial_t \rho + \div (\rho \cdot - \grad \log \rho)  = 0 & \quad \quad  \forall (\vx, t) \in U \times \R_+ ,\\
\partial_\vn \rho = 0 & \quad \quad \forall (\vx, t) \in \partial U \times \R_+ ,\\
\rho(\vx, 0) = \rho_0(\vx) .&
\end{align}
\end{subequations}
The redistribution map is then obtained from:
\begin{subequations} \label{eqGroup:mapDef}
\begin{align}
\rho(\vx, t) =  \rho_0(\vX_B(\vx,t))\det \grad \vX_B(\vx, t) , \\
\vu (\vx, t) = - \grad \log \rho(\vx, t) , \label{eq:defVelo} \\
\partial_t \vX_B + (\vu \cdot \grad) \vX_B = 0 .
\end{align}
\end{subequations}
%In the presence of boundaries, \eqref{eq:defVelo} should be replaced by
%\begin{subequations} \label{eqGroup:defVeloBC}
%\begin{align}
%\div (\rho\vu)  = - \Laplace \rho & \quad \quad \text{ in } U ,\\
%(\rho\vu) \cdot \vn = 0 & \quad \quad \text{ on } \partial U .
%\end{align}
%\end{subequations}

The above flow has the property that $\rho$ follows a heat equation, therefore, the maximum principle guarantees that the density stays bounded away from zero at all times. Furthermore, in the limit as $t \to \infty$, we have that $\rho$ tends to its average $\bar{\rho} = 1$, and hence formally, $\vX_B$ maps the uniform density to the target $\rho_0$. Lastly the diffusion flow consists of an $L^2$-gradient descent of the energy or a $W^2$-gradient descent of the entropy. In fact, an implicit Euler step is equivalent to a minimizing movement scheme in the Wasserstein metric \cite{jordan1998variational}. In this sense, the time-evolution of the map can be seen as an iterative process which contracts to the desired transport map.

\subsection{Energy Estimates}
The density $\rho$ follows an $L^2$-gradient descent of the Dirichlet energy with the usual energy estimate:
\begin{gather} \label{eq:defEnergy}
E(t) = \frac12 \int_U (\rho(\vx, t) - 1)^2 \d x \quad \text{with} \quad \frac{d}{dt} E(t) = - \int_U \left| \grad \rho \right|^2 \d x .
\end{gather}
%We have the usual energy estimate
%\begin{gather}
%\frac{d}{dt} E(t) = - \int_U \left| \grad \rho \right|^2 \d x .
%\end{gather}
Since $\rho$ is a probability density, $\rho - 1$ has zero average for all $t$, so we can apply the Poincar\'{e}-Wirtinger inequality $\| \rho - 1 \|_{L^2} \leq C \| \grad \rho \|_{L^2}$ to get the exponential decay in the $L^2$ energy:
\begin{gather} \label{eq:heatEnergyDecay}
\frac{d}{dt} E(t) \leq - \alpha E(t) ,
\end{gather}%\nopagebreak 
for some constant $\alpha$. %\pagebreak
% avoid having "for some constant a" on next page

In practice, we will compute the backward map to a sufficiently large time $t$ to obtain a transport map between $\rho_0$ and a ``close to uniform'' distribution $\rho_t$. In terms of random variables, we let $Y$ and $Y_t$ be random variables taking value in $U$ with probability densities $\rho_0(\vx)$ and $\rho(\vx, t)$ respectively. We have by construction that $\vX_{[t,0]}(Y_t)$ has density $\rho_0$. Therefore, trivially, the random variables $\vX_{[t,0]}(Y_t)$ converge in distribution to $Y$ as $t \to \infty$. However, we are interested in the above convergence when $Y_t$ is replaced by a fixed random variable with uniform distribution, as this allows us to use $\vX_B$ to redistribute uniform random variables according to density $\rho_0$.

\begin{theorem} \label{thm:distConv}
Let $Z$ be a random variable with uniform distribution on $U$ and also define $Z_t = \vX_{[t,0]}(Z)$ for each $t$. We have that $Z_t$ converges to $Y$ in distribution as $t \to \infty$.
\end{theorem}
\begin{proof}
We have that $Y$ has law $\mu$. We also denote by $\nu_t$ the law of $Z_t$, this is the pushforward measure of $m$ by $\vX_{[t,0]}$. From definition we have
\begin{gather}
\nu_t (A) = m (\vX_{[t,0]}^{-1} (A) ) \quad \text{and} \quad \mu (A) = \mu_t (\vX_{[t,0]}^{-1} (A) ),
\end{gather}
where $\mu_t$ is the law of $Y_t$ and is the pullback measure of $\mu$ by $\vX_{[t,0]}$. Therefore, we have
\begin{gather}
| \nu_t (A) - \mu(A) | = \left| m \left(\vX_{[t,0]}^{-1} (A) \right) - \mu_t \left(\vX_{[t,0]}^{-1} (A) \right) \right| .
\end{gather}

Given that $\vX_{[t,0]}$ is a diffeomorphism on $U$, we have established that
\begin{gather}
Y_t \xrightarrow[]{d} Z \iff Z_t \xrightarrow[]{d} Y,
\end{gather}
since the total variation norms $\| \nu_t - \mu \|_{TV}$ and $\| \mu_t - m \|_{TV}$ are equal for any given $t$.

More precise estimates can also be obtained by looking at the measure densities from taking a Radon-Nikodym derivative with respect to $m$. We have that the measure densities are given by $\d \nu_t = \frac{\d \nu_t}{\d m} \d m$ and $\d \mu = \frac{\d \mu}{\d m} \d m$, with 
\begin{gather}
\frac{\d \mu}{\d m} = \rho_0 = \rho_t (\vX_{[t,0]}^{-1}) \det \grad \vX_{[t,0]}^{-1} \quad \text{and} \quad \frac{\d \nu_t}{\d m} = \det \grad \vX_{[t,0]}^{-1} .
\end{gather}
The $L^2$ distance between the densities is then given by
\begin{gather}
\left\| \frac{\d \nu_t}{\d m} - \frac{\d \mu}{\d m} \right\|_{L^2} = \left( \int_U \left| \rho_0 -  \det \grad \vX_{[t,0]}^{-1} \right|^2 \d m \right)^{\frac12} \nonumber \\
= \left( \int_U \left| \rho_t \circ \vX_{[t,0]}^{-1}  -  1  \right|^2 (\det \grad \vX_{[t,0]}^{-1})^2 \d m \right)^{\frac12} = \left( \int_{\vX_{[t,0]}^{-1}(U)} | \rho_t  -  1  |^2 \det \grad \vX_{[t,0]}^{-1} \d m \right)^{\frac12} \nonumber \\
\leq \| \rho_t - 1 \|_{L^2} \| \det \grad \vX_{[t,0]}^{-1} \|_{L^\infty}^{\frac12} .
\end{gather}

We can further bound $ \det \grad \vX_{[t,0]}^{-1}$ using the maximum principle on $\rho_t$,
\begin{gather}
\det \grad \vX_{[t,0]}^{-1} = \frac{\rho_0(\vx)}{\rho_t ( \vX_{[t,0]}^{-1}(\vx) )} \leq \frac{\sup \rho_0 }{\inf \rho_t} \leq \frac{\sup \rho_0 }{\inf \rho_0} .
\end{gather}

Using the decay rate \eqref{eq:heatEnergyDecay}, this effectively gives us an \emph{a priori} estimate on the $L^2$ density error of  $ \sqrt{2 \frac{\sup \rho_0 }{\inf \rho_0} E(0) } \cdot e^{-\alpha t/2}$ as well as an \emph{a posteriori} estimate $\sqrt{\frac{\sup \rho_0 }{\inf \rho_t} } \cdot \| \rho_t - 1 \|_{L^2}$.
\end{proof}

%\clearpage
\section{Numerical Implementation} \label{sec:densMapNum}

\subsection{Characteristic Mapping method}
In this section we briefly describe the Characteristic Mapping method for a given velocity field $\vu$. We follow the same framework as in \cite{nave2010gradient, seibold2011jet, CME} where more details can be found.

The CM method is essentially composed of a spatial interpolation operator and a time-stepping operator. For the method presented in this paper, we will use a Hermite interpolation in space. Let $\gG$ be a grid on the domain $U$ with grid points $\vx_{i,j}$ and cells $C_{i,j}$. We define $\fV_{\gG, m} \subset C^{m} (U)$ to be the space of Hermite interpolants on $\gG$ of order $2m+1$. The interpolation operator $\He_\gG : C^m (U) \to \fV_{\gG,m}$ is then the projection operator defined by choosing the unique piecewise Hermite polynomial which matches the derivatives $\partial^{\vec{\alpha}} f$ at grid points of $\gG$ for $\vec{\alpha} \in \{0,1, \ldots, m \}^2$. In this paper, we will work exclusively with Hermite linear or cubic interpolants, i.e. $m = 0$ or $1$. For a function $f$ that is at least $2m+2$ times continuously differentiable in space, we have that
\begin{gather} \label{eq:HInterpError}
\left\| \partial^{\vec{\alpha}} \left( f - \He_\gG[f] \right) \right\|_\infty = \bigO \left(\incr{x}^{2m+2-|\vec{\alpha}|} \right) ,
\end{gather}
where $\incr{x}$ is the cell width of the grid $\gG$ and $|\vec{\alpha}|$ denotes $\max (\alpha_1, \alpha_2)$. Inside each cell of $\gG$, $\He_\gG [f]$ is obtained from a tensor product of degree $2m+1$ polynomials in each dimension. From Taylor expansion, we see that the leading order term of the difference between $f$ and this polynomial is $(x - x_i)^{m+1}(x-x_{i+1})^{m+1}$. This means that interpolation or extrapolation of the polynomial at a point $\epsilon < \incr{x}$ away from a grid point yields $\bigO (\epsilon^{m+1 - | \alpha | } \incr{x}^{m+1})$ error.

Finally, it is important to note that we gain an order of accuracy when interpolating the first derivative at cell centres. Indeed, the leading order term in the error in the cell $[x_i, x_{i+1}]$ can be rewritten as $\left( ( x - \frac12 (x_{i}+x_{i+1})  )^2 - \frac14 \incr{x}^2  \right)^{m+1}$. At $x = \frac12 (x_i + x_{i+1})$ this function has vanishing $1^{st}$ derivatives for all $m$. In particular, this means that evaluation of the gradient and first mixed derivatives of a Hermite cubic interpolant is order $\bigO(h_x^4)$ accurate and $\bigO(h_x^2)$ for linear interpolants.

We outline below the algorithm for the time evolution of the characteristic map. We will denote by $\vhX_B \in \fV_{\gG, m}$ the numerical discretization of the exact characteristic map $\vX_B$.

Given discrete time-steps $t_n$, the time-discretization consists in approximating the ``one-step map'' $\vX_{[t_{n+1}, t_n]}$ obtained from integrating the velocity backwards in time. We denote by $\tilde{\vu}$ the approximate numerical velocity. This velocity is defined to be a constant-in-time interpolation for each interval $[t_n, t_{n+1}]$, and is obtained from approximating the true velocity at time $t_{n+1}$. The one-step map is defined using an Euler step:
\begin{gather} \label{eq:defOneStepNum}
\tilde{\vX}_{[t_{n+1}, t_n]}(\vx) = \vx - \incr{t} \tilde{\vu}(\vx) ,
\end{gather}
and the characteristic map is updated by taking the Hermite interpolant of the composition
\begin{gather} \label{eq:mapStepping}
\vhX^{n+1}_B = \He_\gG \left[ \vhX^{n}_B \circ \tilde{\vX}_{[t_{n+1}, t_{n}]}  \right] .
\end{gather}

The definition of $\tilde{\vu}$ is presented in the next section. We note here that once $\tilde{\vu}$ is defined for the interval $[t_n, t_{n+1}]$, the computation of $\tilde{\vX}_{[t_{n+1}, t_n]}$ can be split up into time-subintervals using the same $\tilde{\vu}$. This allows us to pick a small enough time step for the map updates to satisfy the CFL condition while reducing frequency of the calculations for $\tilde{\vu}$.

The error of the map evolution is globally first order in time. The contribution of spatial error to the time stepping comes from the projection of the map composition on the space of Hermite interpolants $\fV_\gG$. The following lemma provides an estimate on this error.

\begin{lemma} \label{lem:stepError}
For $f : U \to \mathbb{R}$ smooth and $\vT : U \to U$ a diffeomorphism with $\vT - \vI = \bigO (\epsilon)$, that is $\vT$ is an order $\epsilon$ perturbation of the identity map. We have that, in the limit of small $\epsilon$,
\begin{gather}
\left\| \partial^{\vec{\alpha}} \left( \He_\gG [ f \circ \vT ] - \He_\gG [f] \circ \vT \right) \right\|_\infty = \bigO (\epsilon \incr{x}^{2m+1-|\vec{\alpha}|} )
\end{gather}
\end{lemma}

\begin{proof}
We write $\vT$ as $\vT(\vx) = \vx + \epsilon \vv(\vx)$. We consider first the Taylor expansion of $f$ at $\vx$:
\begin{gather}
f( \vx + \epsilon \vv(\vx) ) = f(\vx) + \sum_{i \in \{1, 2\} } \epsilon v^i(\vx) \partial_i f(\vx) + \sum_{i,j \in \{1, 2\} } v^i(\vx) v^j(\vx)  \frac{\epsilon^2}{2} \partial_{i,j} f(\vx)  + H.O.T. .
\end{gather}

Applying $\He_\gG$ to $f \circ \vT$ and to $f$, we get
\begin{subequations} \label{eqGroup:compareMapSteps}
\begin{align}
\He_\gG [ f \circ \vT ] = \He_\gG [f] + \sum_{i \in \{1, 2\} }  \epsilon \He_\gG [v^i \partial_i f  ] + \frac{\epsilon^2}{2} \He_\gG [v^i v^j \partial_{i,j} f ]  + H.O.T. ,  \label{eq:projComp} \\
\He_\gG [ f] \circ \vT  = \He_\gG [f] + \sum_{i \in \{1, 2\} }  \epsilon  v^i \partial_i \He_\gG [ f]  + \frac{\epsilon^2}{2} v^i v^j  \partial_{i,j} \He_\gG [f] + H.O.T. . \label{eq:compProj}
\end{align}
\end{subequations}

We omit the interpolation on \eqref{eq:projComp} for all order $\epsilon$ and higher terms as replacing the interpolant with the interpolated function contributes a $\bigO (\epsilon^{|\vec{\alpha} | } \incr{x}^{2m+2})$, $ |\vec{\alpha} | \geq 1$ term which we absorb in the higher order terms. Similarly, we can use a Taylor expansion of $\He_G[f]$ in \eqref{eq:compProj} since the error incurred from extending an order $| \vec{\alpha} |$ derivative of the Hermite interpolant outside a cell is of order $\incr{x}^{2m+2-|\vec{\alpha} |}$ and hence can also be absorbed in the higher order terms. Taking the difference of the two equations in \eqref{eqGroup:compareMapSteps}, we get
\begin{gather}
\He_\gG [ f \circ \vT ] - \He_\gG [ f] \circ \vT = \epsilon \partial_i (f - \He_\gG [ f]) v^i  + \frac{\epsilon^2}{2} \partial_{i,j} (f -\He_\gG [f]) v^i v^j  + H.O.T. .
\end{gather}

In light of \eqref{eq:HInterpError}, we have that
\begin{gather}
\He_\gG [ f \circ \vT ] - \He_\gG [f] \circ \vT = \bigO (\epsilon \incr{x}^{2m+1} ) .
\end{gather}
and therefore
\begin{gather}
\left\| \partial^{\vec{\alpha}} \left( \He_\gG [ f \circ \vT ] - \He_\gG [f] \circ \vT \right) \right\|_\infty = \bigO (\epsilon \incr{x}^{2m+1-|\vec{\alpha}|} )
\end{gather}
\end{proof}

\subsection{Diffusion Flow Velocity} \label{sec:diffFlow}
The velocity $\vu = - \grad \log \rho$ is used to evolved the characteristic map $\vX_B$ whereas the density $\rho$ is a volume form obtained from pullback by $\vX_B$. For the numerical method, we will use a natural staggered grids approach for the discretization of these two quantities in order to obtain a spatially compact scheme. Similar primal-dual grids approaches for solving hyperbolic and parabolic equations using Hermite interpolation have been explored in \cite{appelo2012advection, appelo2018hermite, hagstrom2015solving} where the $H^{m+1}$-seminorm decreasing property of Hermite interpolation was used to design stable methods with high order accuracy.

Here, we use the grid $\gG$ for the definition of the characteristic map $\vX_B$ and the velocity field $\vu$. We define $\gD$ to be the staggered grid of $\gG$ with grid points placed at the cell centers of $\gG$, where the density at time $t_n$ is sampled. We define the grid function $\rho^n_\gD$ to be the evaluation of the following function on the grid $\gD$:
\begin{gather} \label{eq:defNumDens}
\rho^n(\vx) = \rho_0(\vhX^n_B(\vx)) \det \grad \vhX^n_B(\vx) .
\end{gather}
Since the evaluation of $\det \grad \vhX^n_B(\vx)$ occurs at the cell centers of $\gG$, we gain an order of accuracy on the gradient. Therefore, $\rho^n_\gD$ is accurate to order $\bigO (\incr{x}^{2m+2})$ when exact map values are provided. This staggered grid approach is similar in spirit to the primal-dual grid method developed by Appelo et al. in \cite{appelo2018hermite} to solve the scalar wave equation. In that method, a full time-step update of the displacement function goes through two half-step integrations of the velocity function, where in each half-step, the velocity function is computed on a grid dual to the one where the previous displacement function was defined. In that case, the smoothing property of the Hermite interpolation from primal to dual grid yielded stable schemes using very high order interpolation.

In the method proposed here, the velocity field is defined from $\rho^n_\gD$ by taking the log-gradient. In order to avoid the time step constraint of an explicit heat step, we compute the velocity corresponding to an implicit Euler step of the heat equation. 
\begin{subequations} \label{eqGroup:defNumVelo}
\begin{align} 
\tilde{\rho}_\gD = (I - \incr{t} \Laplace )^{-1} \rho^n_\gD , \label{eq:implHeat} \\
\tilde{\vu}(\vx) = - \grad \He_\gD \left[ \log \tilde{\rho}_\gD \right] . \label{eq:stepVeloDef}
\end{align}
\end{subequations}

Indeed, the velocity field $-\grad \log \rho$ extracted from the heat equation is stiff in time and therefore the characteristic ODEs \eqref{eqGroup:charCurve} are stiff. A fully implicit time-stepping method would be complicated and costly due to the coupling of $\rho$ and $\vX_B$, instead, we replace the true velocity $\vu$ with the above constant-in-time implicit Euler approximation. The approach can therefore be thought of as a time-regularization of the characteristic ODEs using a relaxation which is consistent with the underlying heat equation. As reference and clarification, we summarize redistribution algorithm in the pseudocode \ref{alg:heat}.

%
%
%\begin{enumerate}
%\item[1.] Initialize $\vhX_{[t_0, 0]}(\vx) = \vx$ at $t_0 = 0$.
%\item[2.] At time $t_n$, define $\rho^n(\vx)$ according to \eqref{eq:defNumDens}.
%\item[3.] Define $\tilde{\vu}(\vx)$ according to \eqref{eqGroup:defNumVelo}.
%\item[4.] Define the one-step map as a backward Euler step: $\vhX_{[t_{n+1}, t_n]}(\vx) = \vx - \incr{t} \tilde{\vu}(\vx)$. The interval $[t_n, t_{n+1}]$ can be subdivided to satisfy the CFL condition.
%\item[5.] Update the map $\vhX_{[t_{n+1}, 0]} = \He_\gG \left[ \vhX_{[t_n, 0]} \circ   \vhX_{[t_{n+1}, t_n]} \right]$, and repeat from step 2.
%\end{enumerate}
%\vspace*{10pt}

\begin{algorithm}
  \caption{CM method for the heat equation on the space of probability densities.
    \label{alg:heat}}
  \begin{algorithmic}[1]
    \Require{Initial density $\rho_0$, staggered grids $\gG$ and $\gD$, time step $\incr{t}$, final time $T$}
    \Function{HeatFlowMap}{$\rho_0$, $[0,T]$}
      \State Initialize $t \gets 0$, $\vhX_B \gets {\bm {id}}$ \Comment{$\bm{id}$ is the identity map}
      \While{$t < T$}
      	\Let{$\rho^n_\gD$}{$\left[ \rho_0(\vhX_B) \det \grad \vhX_B \right]_\gD$} \Comment{evaluation of \eqref{eq:defNumDens} on $\gD$}
      	\Let{$\tilde{\rho}_\gD$}{$(I - \incr{t} \Laplace )^{-1} \rho^n_\gD$} 
      	\Define{$\vhX_{[t+\incr{t}, t]}(\vx)$}{$\vx + \incr{t} \grad \He_\gD \left[ \log \tilde{\rho}_\gD \right](\vx)$}  \Comment{combining \eqref{eq:stepVeloDef} and \eqref{eq:defOneStepNum}}
      	\Let{$\vhX_B$}{$\He_\gG \left[ \vhX_B \circ \vhX_{[t +\incr{t}, t]} \right]$}
      	\Let{$t$}{$t+\incr{t}$}
      \EndWhile
      \State \Return{$\vhX_B$}
    \EndFunction
  \end{algorithmic}
\end{algorithm}

Using the velocity defined in \eqref{eqGroup:defNumVelo}, we have that the one-step map in \eqref{eq:defOneStepNum} has $\bigO (\incr{t}^2 + \incr{t} \incr{x}^3)$ local truncation error for Hermite cubic interpolants and $\bigO (\incr{t}^2 +\incr{t} \incr{x}^2)$ for linear. Indeed, the Hermite grid data for the linear interpolation have $\bigO (\incr{x}^2)$ error due to being evaluated at cell centres. For cubic interpolation, the derivative data of the velocity field are only available to $\bigO(\incr{x}^2)$, hence the $\bigO(\incr{x}^3)$ spatial error for the one-step map. Given the accuracy of the above one-step maps, we have from lemma \ref{lem:stepError} that the global truncation errors are $\bigO (\incr{t} + \incr{x})$ for linear interpolants, and $\bigO(\incr{t} + \incr{x}^2)$ for cubics.

We apply the redistribution algorithm on a toy problem to illustrate the $L^2$ energy decay. For this test, we take the domain $U$ to be the flat torus $[0,1] \times [0,1]$. The target density $\rho_0$ is concentrated in a band of width $w = 0.15$ around a circle of radius $r = 0.25$ centred at $\vx_c = (0.5, 0.5)$. The minimum density is set at $0.75$.
\begin{subequations} \label{eqGroup:defRho0}
\begin{gather} 
\rho_0 (\vx) = 1 + 0.25 \eta (\vx)  \\
\eta(\vx) = c ( \eta_0 (| \vx - \vx_c | - r) - \bar{\eta}_0 )  \label{eq:defDensVar} \\
\eta_0(s) = \exp ( -(1-(2s/w)^2)^{-1} ) 
\end{gather}
\end{subequations}
Here, the constant $\bar{\eta}_0$ is the average of $\eta_0 (| \vx - \vx_c | - r)$ and $c$ is chosen so that $\min_\vx \eta(\vx) = -1$.

We test the CM method for density redistribution by running the algorithm described in section \ref{sec:densMapNum} to time $1$ using Hermite cubic interpolation on various grids of size $N$ and track the decay of $\mathcal{E}$, $\incr{t}$ is chosen to be $0.1/N$. The results are shown in figure \ref{fig:redist}, the backward map deforms the domain to concentrate in the selected annulus, the residual energy also exhibits the exponential decay from the diffusion equation.

\begin{figure}[h]
\centering
\begin{subfigure}[t]{0.4\linewidth}
%\captionsetup{justification = centering, margin = {-0.125cm, -0.6cm}}
\centering
\includegraphics[width= 0.8\linewidth]{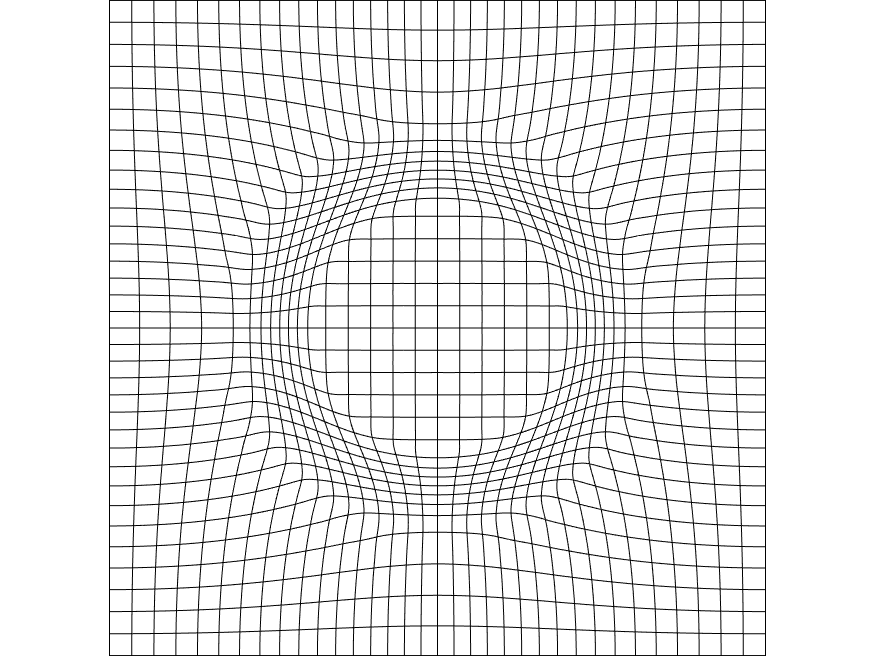}
\caption{Domain deformation from $\vhX_{[1,0]}$.}
\label{subfig:map}
\end{subfigure}
%\hspace*{1.25cm}
\begin{subfigure}[t]{0.4\linewidth}
\centering
\includegraphics[width= 0.8\linewidth]{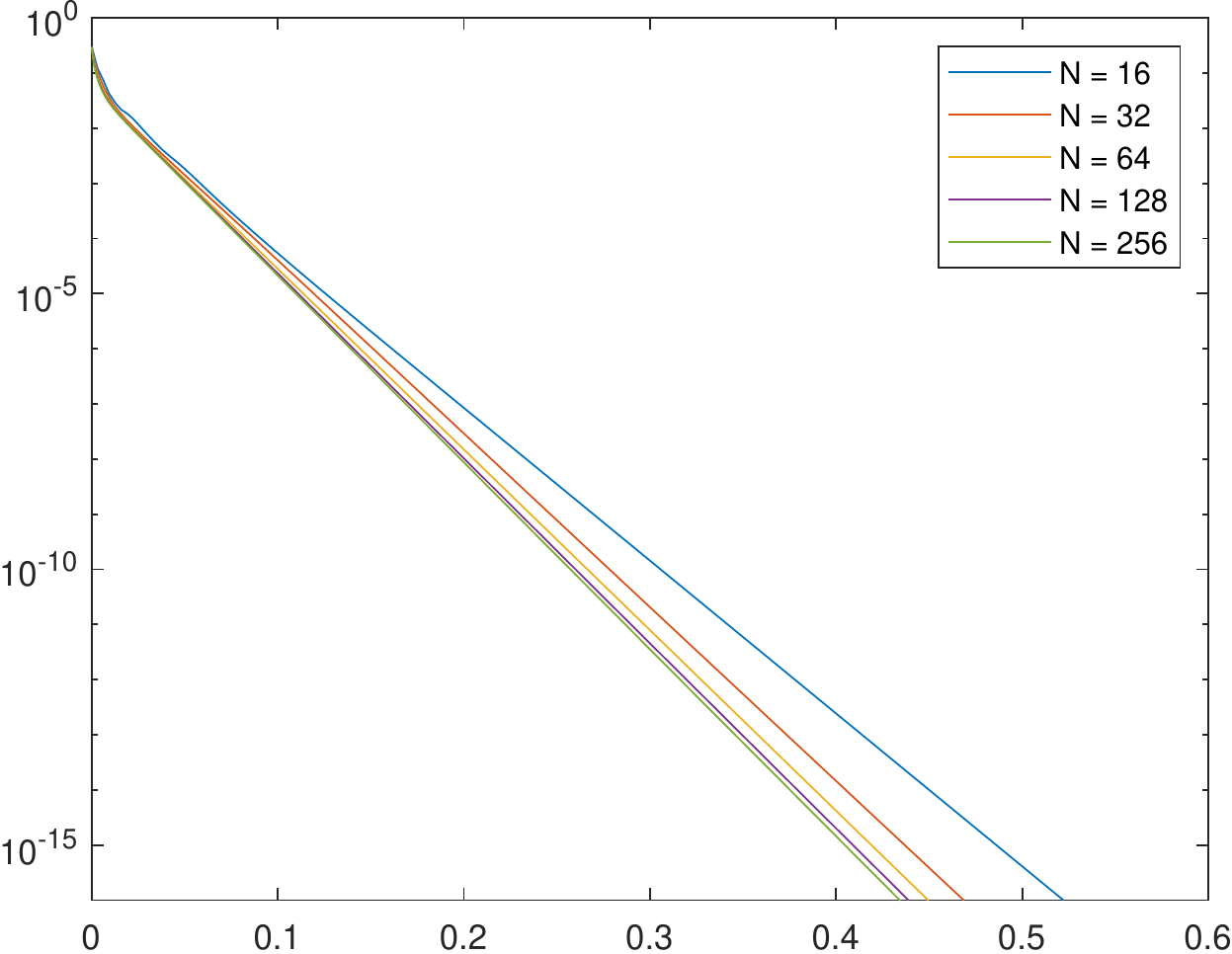}
\caption{Decay of $\mathcal{E}$ for various grids.}
\label{subfig:L2ECubic}
\end{subfigure}
\caption{Diffusion flow density redistribution using CM method.}
\label{fig:redist}
\end{figure}

\section{Application to Surface Advection} \label{sec:surfAdv}

In this section, we apply the density transport method described in section \ref{subsec:densityRedist} to the advection of parametric curves and surfaces in 3-dimensional ambient space. The density transport map will be used to maintain the proper sampling of the moving surface as it stretches and shrinks during its evolution.

\subsection{Evolution of Parametric Surfaces using CM Method}
We briefly describe here the algorithm used to evolve parametric surfaces under a given velocity field in 3D.

Let $\Omega$ be a 3-dimensional domain. For simplicity we assume $\Omega = \mathbb{T}^3$ the flat 3-torus. Let $\vv : \Omega \times \mathbb{R}_+ \to \mathbb{R}^3$ be a given velocity field. Here we assume that $\vv$ is smooth and divergence-free so that the domain transformation it generates is a smooth diffeomorphism of $\Omega$ for all times. We have the same characteristic structure as described in section \ref{subsec:CM} and in \cite{CM, CME}. We denote by $\vPhi_F$ and $\vPhi_B$ the forward and backward characteristic maps, they satisfy the following equations:
\begin{subequations} \label{eqGroup:forwardFlow}
\begin{align} 
\partial_t \vPhi_F (\vp, t) = \vv (\vPhi_F(\vp, t),t) &  \quad \quad  \forall (\vp, t) \in \Omega \times \R_+ , \\
(\partial_t + \vv \cdot \grad )  \vPhi_B (\vp, t) = 0 &  \quad \quad  \forall (\vp, t) \in \Omega \times \R_+ , \\
\vPhi_F(\vp, 0) = \vPhi_B(\vp, 0) = \vp &  \quad \quad  \forall \vp \in \Omega .
\end{align}
\end{subequations}

The diffeomorphisms are the forward and backward flow-maps of the velocity field $\vv$ and have the characteristic structure
\begin{gather}
\vPhi_F(\gamma(0), t) = \gamma(t), \quad \vPhi_B(\gamma(t), t) = \gamma(0) \quad \quad \forall t \in \R_+ ,
\end{gather}
for any characteristic curve $\gamma$ satisfying
\begin{gather}
\partial_t \gamma (t) = \vv (\gamma(t),t) .
\end{gather}
The forward and backward maps are inverse transformations for all times, i.e. 
\begin{gather} \label{eq:invMaps}
\vPhi_F( \vPhi_B(\vp, t), t) = \vPhi_B( \vPhi_F(\vp, t), t) = \vp \quad \quad \forall t \geq 0 .
\end{gather}

We denote by $\vPhi_{[t_1, t_2]}$ the forward map in the time interval $[t_1, t_2]$ with $\vPhi_F := \vPhi_{[0, t]}$, and we use $\vPhi_{[t_2, t_1]}$ to denote the backward map in the same time-interval, with $\vPhi_B := \vPhi_{[t, 0]}$. The notation $\vphi$ is used for the numerical discretization of the $\vPhi$ maps. The time-stepping can be summarized as
\begin{subequations} \label{eqGroup:FBsteps}
\begin{align}
\vphi_{[0, t + \incr{t}]} = \He_\gM \left[\vphi_{[t, t + \incr{t}]}  \circ \vphi_{[0, t]} \right] , \label{eq:stepMapForw} \\
\vphi_{[t + \incr{t}, 0]} = \He_\gM \left[\vphi_{[t, 0]}  \circ \vphi_{[t + \incr{t}, t]} \right] .
\end{align}
\end{subequations}
for some grid $\gM$ on $\Omega$. The identity $\vphi_{[t + \incr{t}, t]} = \vphi_{[t, t + \incr{t}]}^{-1}$ can be guaranteed to high precision by employing higher order ODE integration schemes, and hence the error on the property
\begin{gather}
\vPhi_F \circ \vPhi_B = \vPhi_B \circ \vPhi_F = \vI .
\end{gather}
stems mainly from the representation quality of the interpolation operator $\He_\gM$, which we will control using a dynamic remapping technique described below.

It is well-known that the flow-map possesses a semigroup property which allows for the decomposition of a long-time map into several submaps; this can be used to achieve high resolution representation of the deformation using coarse grid computations. We use the same time-decomposition strategy as in \cite{CME} to represent the global-time maps:
\begin{subequations} \label{eqGroup:submaps3D}
\begin{align}
\vphi_F = \vphi_{[0, t]} = \vphi_{[\tau_{n-1}, t]} \circ \vphi_{[\tau_{n-2}, \tau_{n-1}]} \circ \dots \circ \vphi_{[0, \tau_1]} ,  \\
\vphi_B = \vphi_{[t, 0]} = \vphi_{[\tau_1, 0]}  \circ \dots \circ \vphi_{[\tau_{n-1}, \tau_{n-2}]} \circ \vphi_{[t, \tau_{n-1}]} .
\end{align}
\end{subequations}

The inverse property \eqref{eq:invMaps} is controlled numerically by the time-stepping \eqref{eqGroup:FBsteps}. This gives us an alternative \emph{a posteriori} error estimate for choosing the remapping times $\tau_i$. In \cite{CME}, in the context of the incompressible Euler equations, the remapping time is chosen to the be first time the volume preservation error of the map exceeds some threshold. For surface advection, the velocity is not strictly constrained to divergence free, and volume preservation is not a main concern. Therefore we use the following composition error as remapping criterion.
\begin{gather}
\epsilon = \max \left( \| \vphi_{[\tau_{i-1}, t]} \circ \vphi_{[t, \tau_{i-1}]} - \vI \|_\infty, \| \vphi_{[t, \tau_{i-1}]} \circ \vphi_{[\tau_{i-1}, t]} - \vI \|_\infty  \right) .
\end{gather}
We then define $\tau_i$ to be the first time $t$ where the above error exceeds some chosen threshold.

The forward and backward characteristic maps give us solution operators to the advection problem. Let $S_0 \subset \Omega$ be an initial curve or surface moving following the velocity field $\vv$. Let $S_t$ be the surface at time $t$. Using the characteristic maps, we have two equivalent definitions for $S_t$, one implicit and the other, explicit:
\begin{subequations} \label{eqGroup:defSurf}
\begin{align}
& S_t = \left\{ \vp \in \Omega \; | \; \vPhi_B(\vp, t) \in S_0    \right\} ,  \label{eq:implicitSurfDef}\\
\text{or} \quad & S_t = \left\{ \vPhi_F(\vp, t) \; | \; \vp \in S_0    \right\}. \label{eq:explicitSurfDef}
\end{align}
\end{subequations}

Equation \eqref{eq:implicitSurfDef} uses the same Eulerian definition as in level-set methods. When $S_0$ is expressed as the zero-level-set of a function $\psi$, we have that $S_t$ is given by the zero-level-set of the advected function $\psi \circ \vPhi_B$. This approach is studied in \cite{CM}, a review of recent advances in level-set methods can be found in \cite{gibou2018review}.

Equation \eqref{eq:explicitSurfDef} is a Lagrangian definition, where surfaces are defined explicitly through a parametrization function. This is suitable for curves and surfaces which do not admit a level-set representation, for instance for open curves or non-orientable surfaces. Let $U$ be the parameter space and $\vP_0 : U \to S_0 \subset \Omega$, a regular parametrization of the surface at time 0, that is, we assume the mapping between $U$ and $S_0$ to be a diffeomorphism. The parametrization of $S_t$ is then given by
\begin{gather} \label{eq:evolveParam}
\vP_t = \vPhi_{[0, t]} \circ \vP_0 .
\end{gather}

Numerically speaking this method evolves the parametrization function in time using the solution operator $\vphi_F$. Compared to traditional methods where the surface is sampled and sample points evolved individually, the CM method provides a functional definition of the parametrization function defined everywhere in $U$. Similar to the backward map method in \cite{CM}, this approach also provides arbitrary resolution of the parametrization function $\vP_t$, and hence of $S_t$. The surface $S_t$ can be arbitrarily sampled at any time by evaluating the pushforward operator $\vphi_F$.

\subsection{Equiareal Redistribution} \label{subsec:equiRedist}
A curve or surface $S_t$ can be sampled by choosing sample points $\vy_i \in U$ and evaluating $\vp_i(t) = \vP_t(\vy_i)$ to represent $S_t$ discretely. However, a uniform distribution of $\vy_i$ in $U$ does not necessarily lead to well-distributed marker points $\vp_i$ on $S_t$. Indeed, as the surface stretches and deforms under the flow, some regions may expand and become sparsely sampled. We can quantify this by evaluating the area-element $A_t$ from the first fundamental form $\fundI_t = \grad \vP_t^T \grad \vP_t$.
\begin{gather} \label{eq:defAreaElement}
A_t = \sqrt{\det \fundI_t } = \sqrt{ \det \grad \vP_0^T \grad \vPhi_F^T \grad \vPhi_F \grad \vP_0 } .
\end{gather}

We can compute the time-evolution of $A_t$. Using
\begin{gather}
\partial_t \fundI_t = \grad \vP_t^T  (\grad \vv^T + \grad \vv) \grad \vP_t ,
\end{gather}
and applying Jacobi's rule for derivatives of matrix determinants. We have
\begin{gather}
\partial_t A_t = \frac12 A_t \tr \left( \fundI^{-1} \partial_t \fundI  \right) = \frac12 A_t \tr \left( (\grad \vP_t^T \grad \vP_t)^{-1}  \grad \vP_t^T  (\grad \vv^T + \grad \vv) \grad \vP_t  \right) .
\end{gather}
Using the cyclic property of the trace operator, we can rewrite this as
\begin{gather}
\partial_t A_t = A_t \tr \left( \grad \vP_t  (\grad \vP_t^T \grad \vP_t)^{-1}  \grad \vP_t^T \grad \vv   \right) = A_t \tr \left( \Pi_{S_t} \grad \vv   \right) =  A_t \Div|_{S_t} \vv ,
\end{gather}
where $\Pi_{S_t} = \grad \vP_t  (\grad \vP_t^T \grad \vP_t)^{-1} \grad \vP_t^T$ is the orthogonal projection operator onto the tangent space of $S_t$. Therefore, the area-element grows exponentially at rate $\Div|_{S_t} \vv$ corresponding to the divergence of the flow along the tangent space of $S_t$.

The approach we take to control the growth of the area-element is to modify the parametrization function $\vP_t$ by applying a transformation on the parametric space $U$. We take $A_t$ as a probability density on $U$ and apply the method described in section \ref{subsec:densityRedist} to obtain a mapping $\vX_B : U \to U$ which pushes the uniform density to the density $A_t$. The modified parametrization
\begin{gather}
\vQ_t = \vP_t \circ \vX_{[t, 0]} : U \to S_t \subset \Omega
\end{gather}
should then have area-element $\sqrt{\det \fundI }$ equal to a constant. That is to say, if the points $\vy_i$ are uniformly distributed in $U$, then the sample points $\vp_i$ are uniform on $S_t$. In practice, it is more efficient to run the surface advection and the area redistribution alongside each other, this means that instead of a fixed initial condition for the diffusion equation of section \ref{subsec:densityRedist}, we will use a time-dependent ``initial density'' given by $A_t$. This can be formalized as follows.

We define probability densities on $U$:
\begin{subequations} \label{eqGroup:defDensities}
\begin{align}
& \rho_0(\vx, t) = \frac{1}{\rho_{A}} \sqrt{\det \grad \vP_t^T \grad \vP_t}  , \\
& \rho(\vx, t) = \frac{1}{\rho_{A}}  \sqrt{\det \grad \vQ_t^T \grad \vQ_t} , \\
\text{with} \quad &  \rho_{A} (t) = \int_U \sqrt{\det \grad \vP_t^T \grad \vP_t} \d \vx = |S_t| .
\end{align}
\end{subequations}
We note that both $\rho_0$ and $\rho$ are the area-elements corresponding to the parametrizations $\vP$ and $\vQ$ normalized to a probability density since
\begin{gather}
\rho_A(t) = \int_{\vX_B^{-1}(U)} \sqrt{\det \grad \vP_t^T \grad \vP_t}|_{\vX_B} \det \grad \vX_B \d \vx = \int_U \sqrt{\det \grad \vQ_t^T \grad \vQ_t} \d \vx.
\end{gather}

We have the following equations for their time evolution:
\begin{subequations} \label{eqGroup:densityEvolution}
\begin{align}
& \partial_t \rho_{A} (t) = \int_U \Div|_{S_t} \vv \sqrt{\det \grad \vP_t^T \grad \vP_t} \d \vx = \rho_{A} \int_U \rho_0 \Div|_{S_t} \vv \d \vx, \\
& \partial_t \rho_0(\vx, t) = \frac{1}{\rho_{A}} \Div|_{S_t} \vv \sqrt{\det \grad \vP_t^T \grad \vP_t} - \frac{\rho_0}{\rho_{A}} \partial_t \rho_{A} = \left(  \Div|_{S_t} \vv -  \partial_t \log \rho_{A}  \right) \rho_0(\vx, t)  , \\
& \rho(\vx, t) = \rho_0 \left(\vX_B(\vx, t), t \right) \det \grad \vX_B .
\end{align}
\end{subequations}

Using the results in section \ref{subsec:densityRedist}, we have that the pullback of $\rho_0$ by $\vX_B$ generates a diffusion process in $\rho$. Therefore, the governing equation for the area-element of the redistributed parametrization $\vQ_t$ is
\begin{gather} \label{eq:redistGovEqn}
\partial_t \rho(\vx, t) = \nu \Laplace \rho(\vx, t)  + \partial_t \rho_0 |_{\vX_B} \det \grad \vX_B = \nu \Laplace \rho(\vx, t) + \lambda_{\vv}(\vx, t) \rho (\vx, t) ,
\end{gather}
where $\lambda_{\vv} (\vx, t) = \partial_t \log \rho_0 |_{\vX_B} =  \left( \Div|_{S_t} \vv \right)_{(\vX_B)} -  \int_U \rho_0 \Div|_{S_t} \vv \d \vx $.
%\begin{gather} \label{eq:defSigma}
%\lambda_{\vv} (\vx, t) = \left( \Div|_{S_t} \vv \right)_{(\vX_B)} -  \int_U \rho_0 \Div|_{S_t} \vv \d \vx .
%\end{gather}
The diffusion coefficient $\nu$ is introduced in the redistribution step to increase control over the growth of area density from surface deformation.

%\begin{enumerate}
%\item[1.] Initialize $\vphi_{[0, t_0]}(\vp) = \vp$, $\vhX_{[t_0, 0]}(\vx) = \vx$ at $t_0 = 0$.
%\item[2.] At $t_n$, define $\rho^n_0(\vx) = \left( \int_U A_{t_n} \d x \right)^{-1} A_{t_n}$ using \eqref{eq:defAreaElement} and normalize to a probability density.
%\item[3.] Define $\rho^n(\vx) = \rho_0^n( \vhX_{[t_n, 0]}(\vx) ) \det \grad \vhX_{[t_n, 0]}(\vx)$.
%\item[4.] Define $\tilde{\vu}(\vx)$ according to \eqref{eqGroup:defNumVelo}.
%\item[5.] Update redistribution map $\vhX_{[t_{n+1}, 0]} = \He_\gG \left[ \vhX_{[t_n, 0]} \circ \vhX_{[t_{n+1}, t_n]}   \right]$ using \eqref{eq:defOneStepNum} and \eqref{eq:mapStepping}.
%\item[6.] Advance forward map $\vphi_{[0, t_{n+1}]} = \He_\gM \left[\vphi_{[t_n, t_{n+1}]} \circ \vphi_{[0, t_n]}  \right]$, and repeat from step 2.
%\end{enumerate}

The method is summarized in pseudocode \ref{alg:surfParam}. The notation $\vP_t$ implicitly assumes the definition and computation of the original time-dependent parametrization function \eqref{eq:evolveParam} using the numerical characteristic map $\vphi_F$ of the ambient advection. The computation of $\vphi_F$ is given in\eqref{eq:stepMapForw} and \eqref{eqGroup:submaps3D}. In comparison with the algorithm described in section \ref{sec:diffFlow}, the redefinition of $\rho_0^n$ at each $t_n$ at line 4 corresponds to a source term $\lambda_{\vv} \rho$ in the density evolution arising from the deformation of the surface. A typical energy argument for the reaction-diffusion equation provides some estimates on the evolution of the density. Here we use the special structure of the $\lambda_{\vv}$ term to write a more specific estimate.

\begin{algorithm}
  \caption{CM method for equiareal time-dependent surface parametrization
    \label{alg:surfParam}}
  \begin{algorithmic}[1]
    \Require{Parametrization $\vP_t$, staggered grids $\gG$ and $\gD$, diffusivity $\nu$, time step $\incr{t}$, final time $T$}
    \Function{Reparametrization}{$\vP_t$, $\nu$, $T$}
      \State Initialize $t \gets 0$ ,$\vhX_{[t, 0]} \gets {\bm {id}}$  \Comment{$\bm{id}$ is the identity map}
      \While{$t < T$}
        \Define{$\rho^n_0(\vx)$}{$\frac{1}{\rho_{A}(t)} \sqrt{\det \grad \vP_t^T \grad \vP_t}$} \Comment{from \eqref{eqGroup:defDensities}}
      	\Let{$\vhX_{[t+\incr{t}, t]}$}{HeatFlowMap$(\rho^n_0, [0, \nu \incr{t}])$} \Comment{subroutine \ref{alg:heat} on local-time problem}
      	\Let{$\vhX_{[t+\incr{t}, 0]}$}{$\He_\gG \left[ \vhX_{[t, 0]} \circ   \vhX_{[t +\incr{t}, t]} \right]$}
      	\Let{$t$}{$t+\incr{t}$}
      \EndWhile
      \State \Return{$\vQ_T = \vP_T \circ \vhX_{[T, 0]}$}
    \EndFunction
  \end{algorithmic}
\end{algorithm}

\begin{remark}
The diffusion coefficient $\nu$ is implemented in algorithm \ref{alg:surfParam} as a scaling on the diffusion time for the local time computations, by evolving the diffusion for a total time of $\nu \incr{t}$. The time steps used for the computation of the subroutine \ref{alg:heat} can be adjusted independently of the $\incr{t}$ in algorithm \ref{alg:surfParam} according on the given density $\rho^n_0$. In practice, the same $\incr{t}$ is used in both routines.
\end{remark}

\begin{theorem} \label{thm:areaErrorBound}
Choosing $\nu$ large enough, the $L^2$ distance between the area-element of the $\vQ$ parametrization and the uniform distribution can be controlled to order $\bigO(\nu^{-1})$.
\end{theorem}
\begin{proof}
The algorithm is consistent with the reaction-diffusion equation \eqref{eq:redistGovEqn}. We have that the energy estimate in this case is
\begin{gather}
\frac{d}{dt} E(t) = - \nu \|\grad (\rho -1) \|_{L^2}^2  + \int_U  (\rho -1)^2 \lambda_{\vv} \d \vx + \int_U (\rho -1) \lambda_{\vv} \d \vx \nonumber \\
\leq - \nu \|\grad (\rho -1) \|_{L^2}^2 + \| \lambda_{\vv} \|_\infty \| \rho -1 \|_{L^2}^2 - \int_U \lambda_{\vv} \d \vx ,
\end{gather}
where $\int_U \rho \lambda_{\vv} \d \vx$ vanishes since $\int_U \partial_t \rho \d \vx = 0$ by construction and $\int_U \Laplace \rho \d \vx =0$ with periodic or Neumann boundary conditions.

The integral $\int_U \lambda_{\vv} \d \vx$ acts as a ``persistent'' source term in the energy decay since it corresponds to the density changes the surface deformation applies to the area density independently of its $\rho-1$ deviation from uniform. We have the following bound on this integral:
\begin{gather}
- \int_U \lambda_{\vv} \d \vx = - \int_U \left( \Div|_{S_t} \vv \right)_{(\vX_B)} \d \vx +  \int_U \rho \left( \Div|_{S_t} \vv \right)_{(\vX_B)} \d \vx \\
 = \int_U (\rho -1 )\left( \Div|_{S_t} \vv \right)_{(\vX_B)} \d \vx  \leq \left\|\left( \Div|_{S_t} \vv \right)_{(\vX_B)} \right\|_{L^2} \| \rho - 1 \|_{L^2}  .  \nonumber
\end{gather}

We obtain the following bound on the growth of the $L^2$ norm:
\begin{gather}
\frac{d}{dt} \| \rho - 1\|_{L^2} \leq 2 (- \alpha \nu + \| \lambda_{\vv} \|_\infty  )  \| \rho - 1\|_{L^2} + 2 \left\|\left( \Div|_{S_t} \vv \right)_{(\vX_B)} \right\|_{L^2}  ,
\end{gather}
meaning that the $L^2$ norm can be controlled by
\begin{gather}
\| \rho - 1\|_{L^2} \leq \frac{\beta}{\alpha \nu - \| \lambda_{\vv} \|_\infty} + c \exp \left( -2t \left( \alpha \nu - \| \lambda_{\vv} \|_\infty \right) \right) ,
\end{gather}
where $\beta = \max_t \left\|\left( \Div|_{S_t} \vv \right)_{(\vX_B)} \right\|_{L^2}$ and $\alpha$ the constant from the Poincar\'e inequality.

Therefore, by choosing $\nu > \alpha^{-1} \| \lambda_{\vv} \|_\infty$ sufficiently large, we can guarantee that the $L^2$ norm of the deviation from uniform of the $\vQ$ area-element stays of order $\bigO (\nu^{-1} )$ for all times.
\end{proof}

According to the governing equation \eqref{eq:redistGovEqn}, for larger enough $\nu$, the diffusion should limit the fine scale spatial features present in $\rho$ and hence in the velocity field $-\grad \log \rho$. High frequency modes are generated by $\lambda_\vv$ and hence it is sufficient to compute the local-time deformation map on a grid fine enough to resolve the source term. Similar to the advection problem where the local maps $\vphi_{[\tau_{i-1}, \tau_{i}]}$ are computed on coarse grids which resolve well enough the local-time velocity field, the submaps in the density transport map can also benefit from the computational savings of coarser grids. All characteristic maps involved share the same semigroup structure and can be decomposed into submaps in order to achieve higher spatial resolution at low computational cost. We can therefore apply the above reparametrization algorithm to each subinterval in the submap decomposition \eqref{eqGroup:submaps3D} and obtain the full reparametrization as the composition of all redistribution maps. Combining \eqref{eqGroup:submaps3D}, \eqref{eq:evolveParam} and \eqref{eq:SubmapDecomp}, we have
\begin{gather} \label{eq:movedParamNum}
\vQ_t = \vphi_{[\tau_{n-1}, t]} \circ \vphi_{[\tau_{n-2}, \tau_{n-1}]} \circ \cdots \circ \vphi_{[0, \tau_1]} \circ \vP_0 \circ \vhX_{[\tau_1, 0]} \circ \cdot \circ \vhX_{[\tau_{n-1}, \tau_{n-2}]} \circ \vhX_{[t, \tau_{n-1}]} .
\end{gather}
Numerically, each submap is computed independently, sequentially using algorithm \ref{alg:surfParam}. To be consistent with the reinitialized problem, the input parametrization for the $i+1^{st}$ map is defined to be $\vphi_{[\tau_i, t]} \circ \vQ_{\tau_i}$.

\begin{remark}
For the submap computations, the parametrization of the surface $S_t$ is given by $\vphi_{[\tau_i, t]} \circ \vQ_{\tau_i}$ where $\vQ_{\tau_i}$ is given by \eqref{eq:movedParamNum}. However, to save computational time, for the purpose of computing the density $\rho$, it is sufficient to replace $\vQ_{\tau_i}$ by an interpolant on a fine enough grid similar to the approach in \cite{CM}. This is because the method maintains the sampling density of $\vQ$ near uniform for all times and therefore $\vQ_{\tau_i}$ can be accurately represented by interpolation.
\end{remark}

\section{Numerical Results} \label{sec:numResults}

\subsection{Density Redistribution on Flat Domains}
In this section, we test the redistribution algorithm on an evolving probability density $\rho_0(\vx, t)$ in a flat periodic domain $U$. A redistributed density $\rho$ is obtained from a redistribution map $\vX_B$ computed using algorithm \ref{alg:surfParam} where $\rho_0$ evaluated at line 4 is instead assumed to be given. The probability density $\rho(\vx, t) = \rho_0 (\vX_B(\vx, t), t)$ would then evolve according to a heat equation with source term
\begin{gather}
\partial_t \rho = \nu \Laplace \rho + \rho \, \partial_t \log \rho_0|_{\vX_B}  .
\end{gather}

We define the following density $\rho_0$:
\begin{gather}
\rho_0(x,y,t) = 1 + 0.25 \sin \left( 1.5 \pi t  \right) \eta \left( x + 0.25 \sin \left( 0.5 \pi t  \right) \sin (2 \pi y), y \right) ,
\end{gather}
where $\eta$ is defined in \eqref{eq:defDensVar}. 

The resulting density $\rho_0$ is initially concentrated around a circle of radius $0.25$ and is advected by the volume preserving transformation $(x,y) \mapsto (x + 0.25 \sin \left( 0.5 \pi t  \right) \sin (2 \pi y), y)$. The amplitude of the density is scaled by $ 0.25 \sin \left( 1.5 \pi t  \right)$.

We test the redistribution of this moving density in the time interval $t \in [0, 3]$ with various diffusion coefficients $\nu$ and with various grid sizes for the characteristic map which we represent using piecewise linear interpolation. We ran convergence tests of the maximum $L^2$ error for $t \in [0, 3]$ with respect to the grid size and $\nu$. The results are shown in figure \ref{fig:Conv}. As expected, the maximum $L^2$ norm of $\rho - 1$ is linear with respect to $\incr{x}$, $\incr{t}$ and $\nu^{-1}$.

\begin{figure}[h]
\centering
\begin{subfigure}{0.45\linewidth}
\centering
\includegraphics[width = 0.6\linewidth]{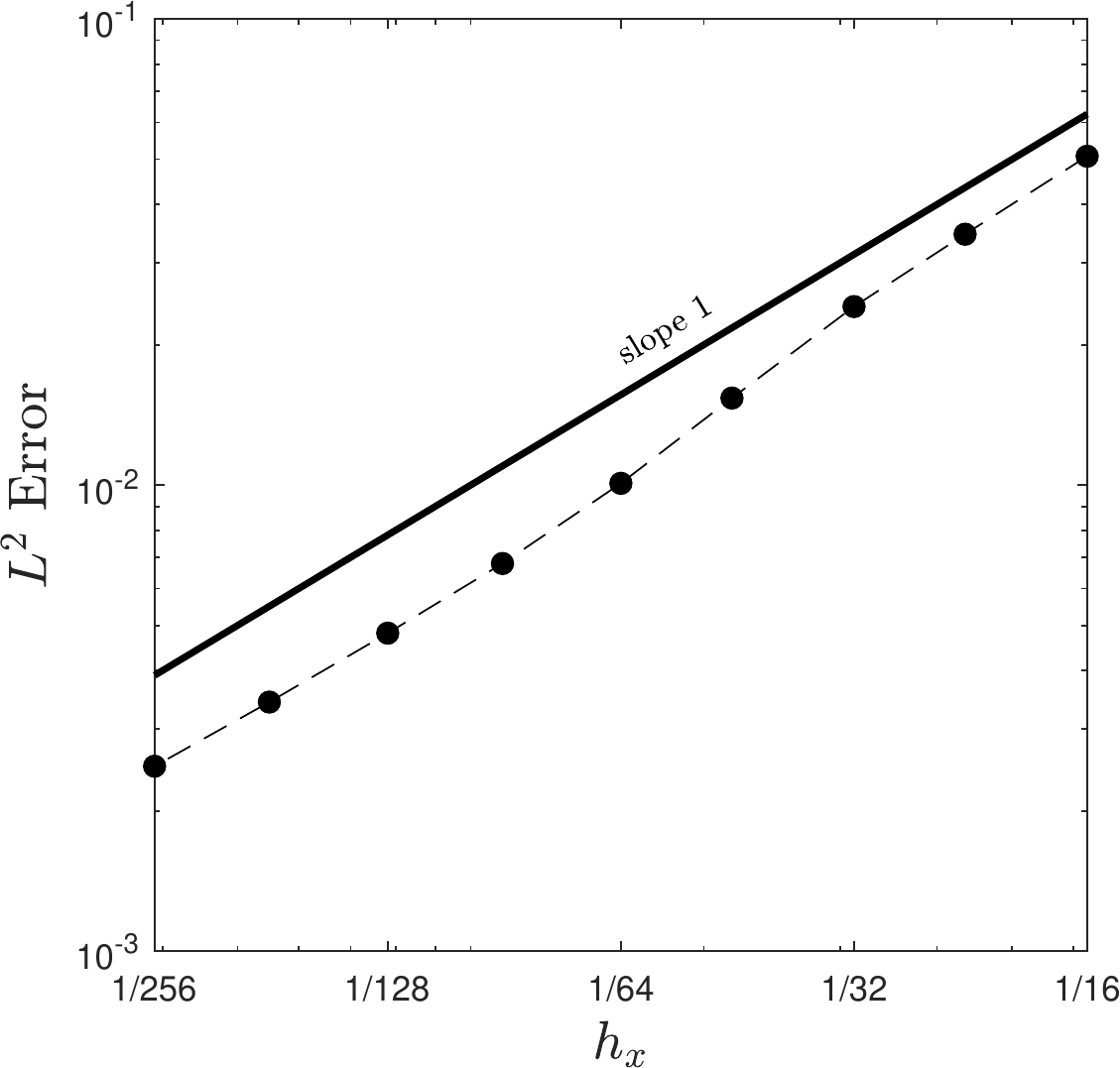}
\caption{$\nu = 10$, $\incr{t} = \incr{x}/4$.}
\label{subfig:gridConv}
\end{subfigure}
\begin{subfigure}{0.45\linewidth}
\centering
\includegraphics[width = 0.6\linewidth]{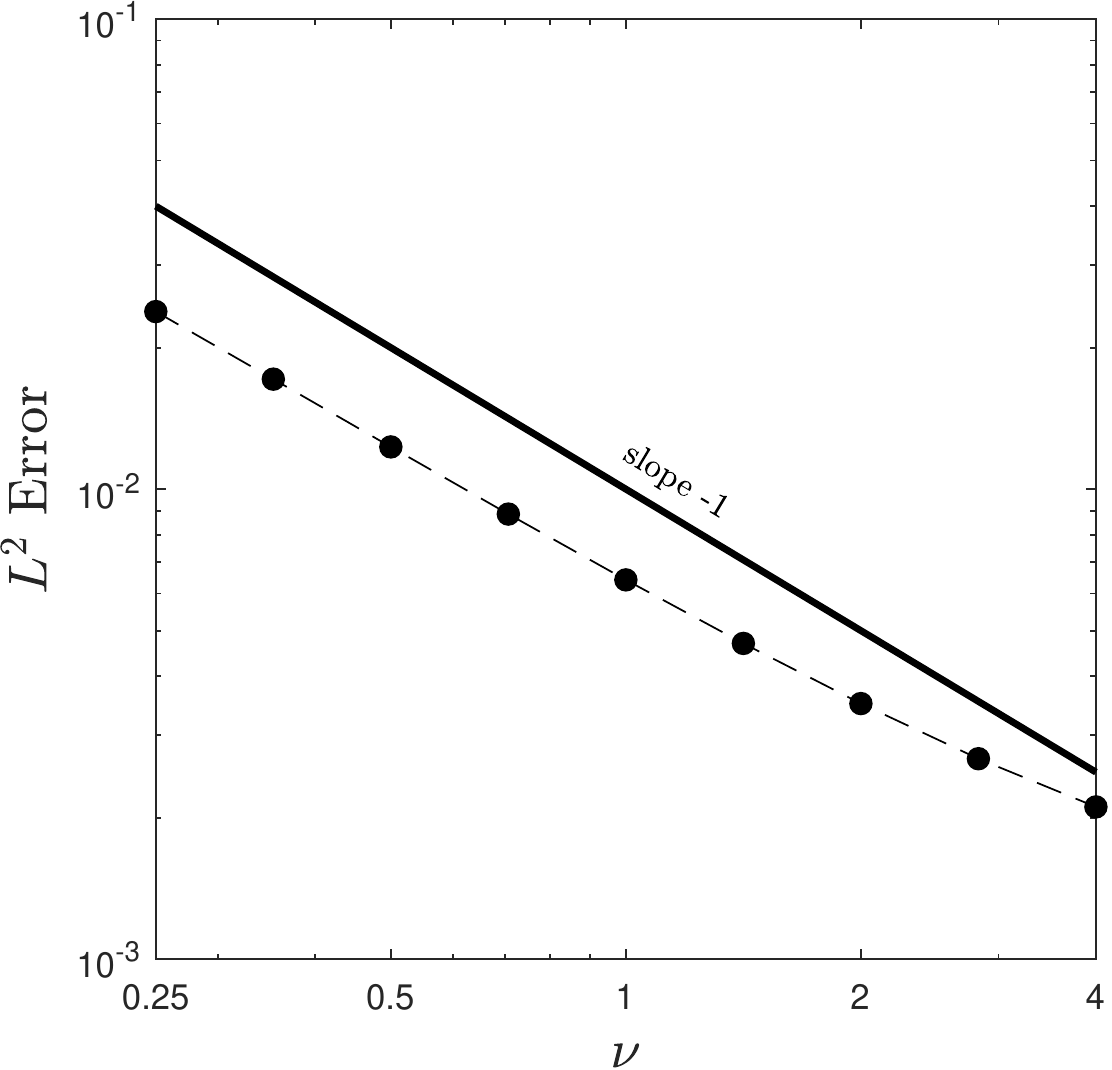}
\caption{$\incr{x} = 1/768$, $\incr{t} = 1/2048$.}
\label{subfig:nuConv}
\end{subfigure}
\caption{$L^2$ error with respect to $\incr{{ }}$ and $\nu$.}
\label{fig:Conv}
\end{figure}

\subsection{Equidistributing Parametrizations of Curves and Surfaces} \label{subsec:equidistCS}
In this section, we apply the redistribution algorithm to the evolution of several curves and surfaces in a 3D ambient flow. Starting from a given initial parametrization, we use the redistribution characteristic map to maintain an equiareal parametrization of the evolving curves and surfaces at all times during the simulations. For the tests in this section, we use the following 3D ambient velocity taken from \cite{leveque1996high}:
\begin{gather} \label{eq:3Dambient}
\vv(\vx, t) = \left( \begin{matrix}
2\cos\left(\frac{\pi t}{P}\right) \left( \sin(\pi x) \right)^2 \sin(2 \pi y) \sin(2 \pi z)\\
-\cos\left(\frac{\pi t}{P}\right) \sin(2 \pi x) \left( \sin(\pi y) \right)^2 \sin(2 \pi z)\\
-\cos\left(\frac{\pi t}{P}\right) \sin(2 \pi x) \sin(2 \pi y) \left( \sin(\pi z) \right)^2
\end{matrix}  \right) 
\end{gather}
defined on a periodic cube $[0,1]^3$, with $\vx = (x,y,z)^T$. $P$ is the period of the velocity field, the deformation it generates reaches its maximum extent at $t= P/2$ then returns to identity at $t=P$. For the following tests, we will choose $P=3$.

The velocity field also has reflection symmetries across the planes $z = y$ and $z = 1-y$, the flow is also planar along these two planes. We know therefore that the flow will not cross these planes and will have mirror symmetric motion on either sides. It follows that any initial curve or surfaces crossing these planes will undergo extensive deformation. Without any maintenance on the parametrization function, one can expect the resulting arclength and area elements to grow exponentially, resulting in poor representation of the curves and surfaces.

For all results in this section, we used a $64^3$ grid with $\incr{t} = 1/96$ for the computation of the forward characteristic map in the 3-dimensional ambient space, we used a $128^2$ grid for the computation of the 1D and 2D redistribution maps on the parametric space. The ambient map uses Hermite cubic interpolation and redistribution maps use linear interpolation. The diffusion coefficient for all redistribution maps were fixed at $\nu = 2$. All characteristic maps use the submap decomposition method and all curves and surfaces are pushed forward using the same forward ambient space map. The computations were carried out on a laptop with an Intel i5-3210 duo-core 2.50 GHz CPU with 8 GB of RAM. The routines are implemented in Matlab with C-Mex subroutines for the interpolation operations. As reference, the 3D ambient characteristic maps was decomposed into 6 submaps in the interval $[0, 1.5]$, each map being stored on hard drive. The total computational time for calculating the ambient characteristic map was 511 seconds.

\subsubsection{Evolution of Curves} \label{subsec:evoCurves}

We apply the redistribution method in 1-dimensional parametric space to maintain an arclength parametrization of curves evolving under the flow given in \eqref{eq:3Dambient}. We use 4 curves in this test, the first three are line segments and the last one is a circle. 

To illustrate the effect of the redistribution, we show in figure \ref{fig:Creso} the final states of the curves at various resolutions (the initial curves and their time evolution are shown in Appendix \ref{sec:appendixCurves}). In figure \ref{fig:Creso}, the parametrization $\vP$ as well as its redistributed version $\vQ$ are drawn using a piecewise linear interpolation on a gradually refined grid from $32$ equidistant grid points to $1024$ grid points. As we can see, in all cases, to capture the features of a given curve, the $\vP$ parametrization requires a roughly $10\times$ finer grid to obtain the same quality as the $\vQ$ parametrization. This is due to the high variations in the speed of the $\vP$ parametrization. Indeed, due to the large distortions created by the ambient flow, some regions of the curves undergo large stretching whereas others are compressed. At $t=1.5$, this results in high amplitude variations in the distances between marker points that were equidistant at $t=0$. This effect can be clearly seen in figure \ref{fig:lengthDist} where the histogram of the cell lengths for the curves are shown. As we can see, the cell lengths of the original parametrization $\vP$ (shown in blue) are rather spread out, with a majority of very short cells covering one part of the curves and few very large cells covering the rest. In terms of sampling, this is suboptimal since the marker points in the oversampled regions are redundant. In contrast, the redistributed parametrizations $\vQ$ (shown in red/orange) have a much more uniform distribution of cell lengths: almost all cell lengths are concentrated around the average, meaning that marker points are uniformly distributed along the curve. We can also measure the uniformity of the marker points distribution quantitatively: table \ref{tab:Cstats} shows the standard deviation, i.e. $\| \rho^n  - 1 \|_{L^2}$, and the median of the normalized area densities for each parametrization as well as the computational time required for generating the redistribution map. Ideally, for a perfectly uniform distribution, the median cell length should be 1 and the standard deviation 0. A median closer to 1 and a smaller standard deviation in the $\vQ$ case indicates that the arclength distance between two sample points are more uniform thereby avoiding the undersampling of the more deformed parts of the curves or redundancy of markers in compressed parts.

\begin{figure}[h]
\centering
\begin{subfigure}{0.4\linewidth}
\centering
\includegraphics[width = 0.8\linewidth]{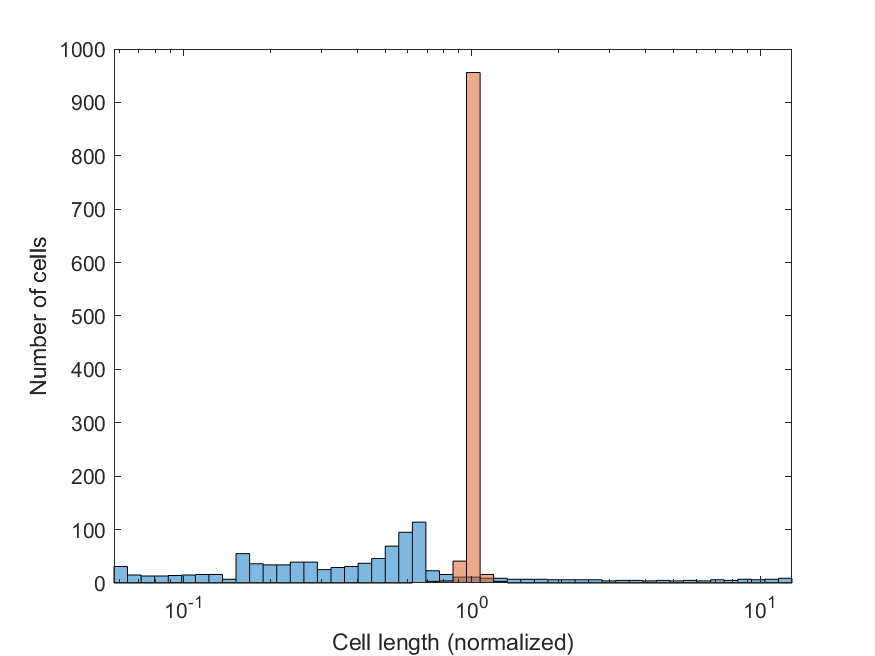}
\caption{Curve 1}
\label{subfig:C1_lendist}
\end{subfigure}
\begin{subfigure}{0.4\linewidth}
\centering
\includegraphics[width = 0.8\linewidth]{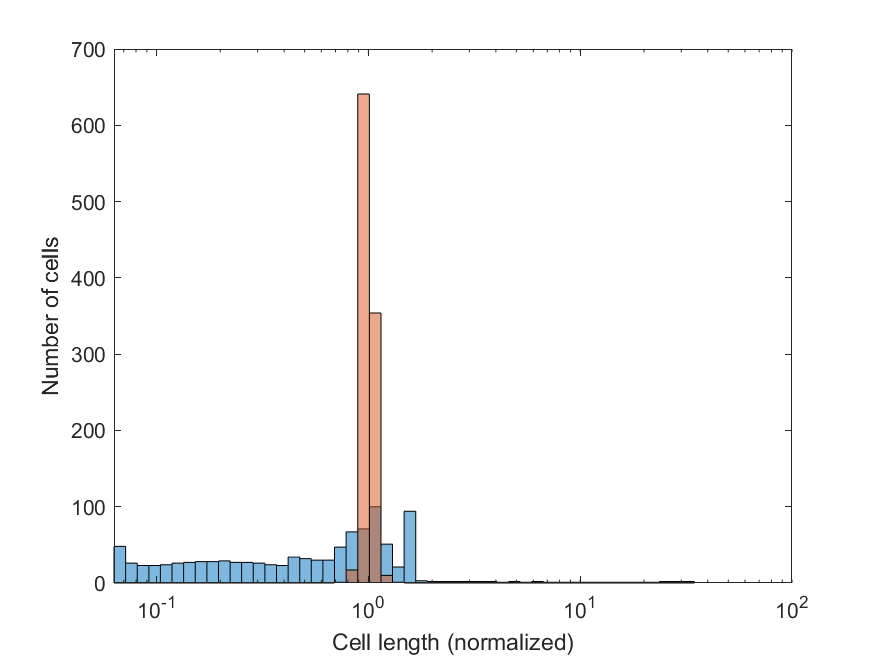}
\caption{Curve 2}
\label{subfig:C2_lendist}
\end{subfigure}
\begin{subfigure}{0.4\linewidth}
\centering
\includegraphics[width = 0.8\linewidth]{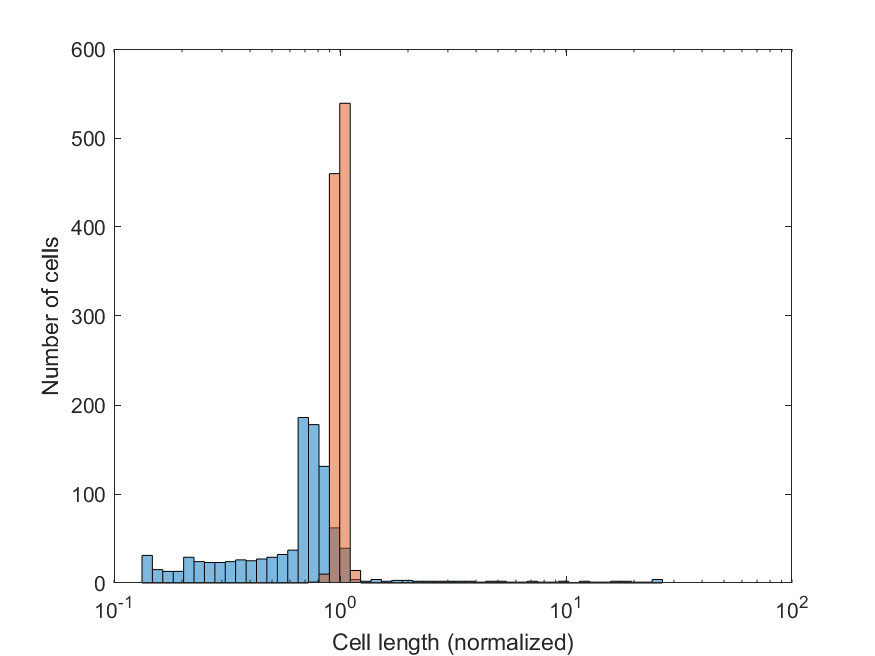}
\caption{Curve 3}
\label{subfig:C3_lendist}
\end{subfigure}
\begin{subfigure}{0.4\linewidth}
\centering
\includegraphics[width = 0.8\linewidth]{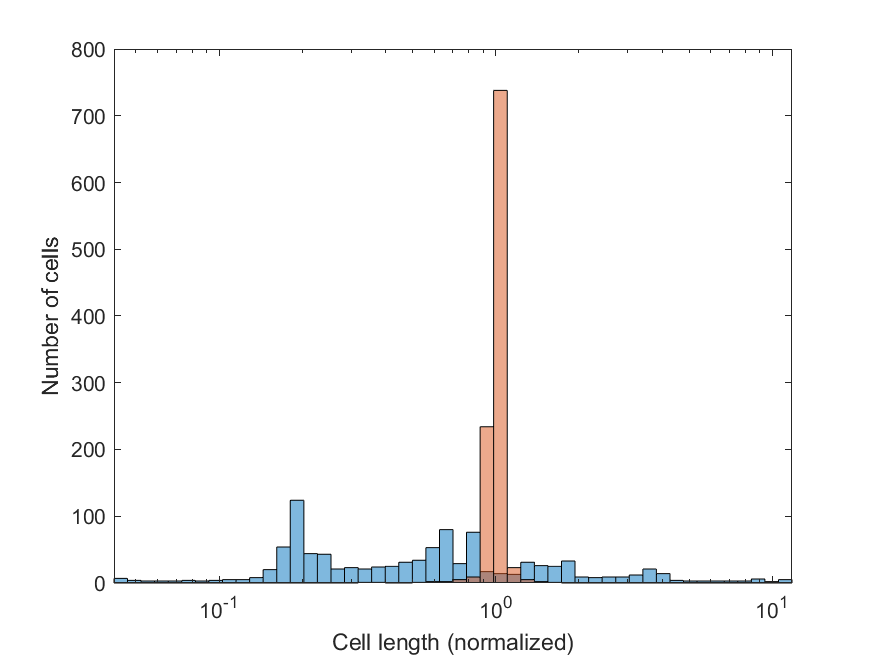}
\caption{Curve 4}
\label{subfig:C4_lendist}
\end{subfigure}
\caption{Arclength distribution of the parametrizations $\vP$ (blue) and $\vQ$ (red)}
\label{fig:lengthDist}
\end{figure}

\begin{figure}[H]
\centering
\includegraphics[width = \linewidth]{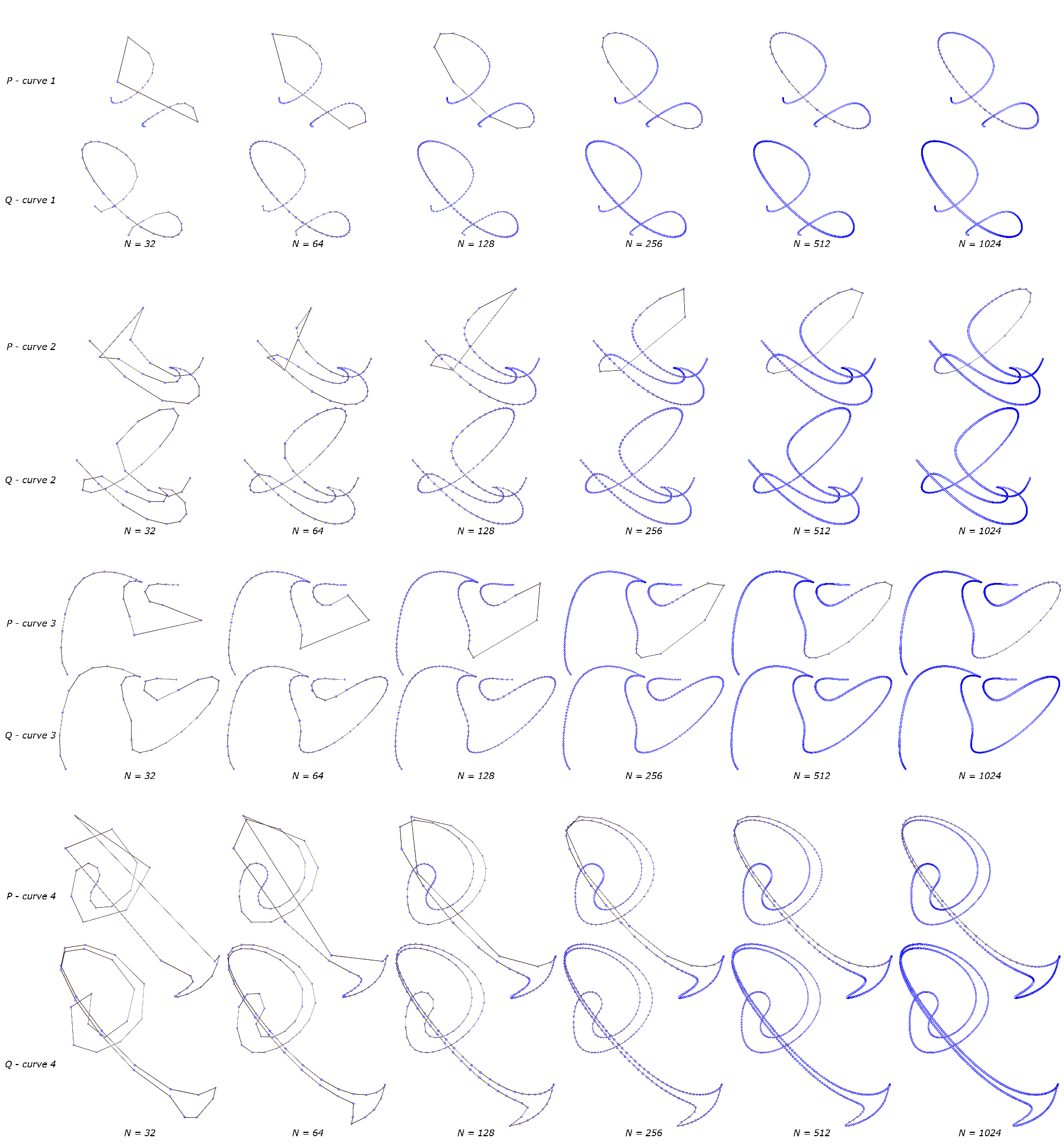}
\caption{$\vP$ and $\vQ$ parametrizations of curves 1 to 4 using gradually finer grids.}
\label{fig:Creso}
\end{figure}

\begin{table}[H] \footnotesize
\begin{subtable}{\linewidth}
\begin{center}
%{\renewcommand{\arraystretch}{1.5}
\begin{tabular}{c | c c c c}
\hline
\hline
Curves & 1 & 2 & 3 & 4 \\
\hline
$\sigma_{\vP}$ & 1.0479 & 0.8214 & 0.7088 & 0.8858 \\
$\sigma_{\vQ}$  & 0.0195 & 0.0318 & 0.0248 & 0.0353 \\
\hline
$M_{\vP}-1$  & -0.5358 & -0.3742 & -0.2927 & -0.4435 \\
$M_{\vQ}-1$  & 0.0013 & 0.0012 & -0.0002 &  0.0005 \\
\hline
\end{tabular}%} \\
\end{center}
\captionsetup{width=.95\linewidth}
\caption{Standard deviation ($\sigma$) and median ($M$) errors of the length density.}
\label{subtab:Cerrors}
\end{subtable}
\newline
\vspace*{16pt}
\newline
\begin{subtable}{\linewidth}
\begin{center}
%{\renewcommand{\arraystretch}{1.5}
\begin{tabular}{c | c c c c}
\hline
\hline
Curves & 1 & 2 & 3 & 4 \\
\hline
Evaluating $\vQ$ & 9.96 s & 9.32 s & 9.61 s & 9.62 s \\
Defining $\rho^n(\vx)$ & 0.04 s & 0.04 s & 0.05 s  & 0.04 s \\
Updating $\vhX_B$ & 0.74 s & 0.70 s & 0.67 s & 0.59 s \\
Number of remappings & 2 & 3 & 3 & 3 \\
\hline
\end{tabular}%} \\
\end{center}
\captionsetup{width=.95\linewidth}
\caption{Total computation times for the evolution of the parametrization $\vQ$.}
\label{subtab:Ctimes}
\end{subtable}
\caption{Parametrizations $\vP$ and $\vQ$ at $t=1.5$.}
\label{tab:Cstats}
\end{table}

\subsubsection{Evolution of Surfaces}
We test the redistribution method on three different topologies for 2-dimensional surfaces: rectangle, torus and cylinder. The parametric spaces are taken to be $U = [0,1]^2$ with  Neumann, periodic and mixed Neumann-periodic boundary conditions respectively. These surfaces will move under the flow \eqref{eq:3Dambient} and we will compute the two parametrizations $\vP$ and $\vQ$ as in the 1D case. For each given time $t$ shown in figures \ref{fig:S1Evolve} to \ref{fig:S3Evolve}, the parametrizations are represented by a linear interpolation on a uniform mesh grid of $512^2$ points. For $\vP$ the grid data is obtained by solving the ODEs forward in time for each grid point, $\vQ$ is obtained by evaluating the redistribution map, $\vP_0$ and the forward ambient characteristic maps at grid points on $U$. In order to illustrate the effect of the redistribution, we sample each parametrization with 200 000 randomly generated marker points. The distributions of these random points over the surfaces are expected to follow the random variables description in section \ref{subsec:densityRedist}. The initial surfaces with marker points are shown in \ref{fig:InitSurfs}, the time evolution of each surface is shown in figures \ref{fig:S1Evolve} to \ref{fig:S3Evolve}. We also show the standard deviation and median of the normalized area densities and the computational times for the redistribution maps in figure \ref{fig:areaDist} and table \ref{tab:Sstats}.

\begin{figure}[h]
\centering
\begin{subfigure}{0.15\linewidth}
\includegraphics[width = \linewidth]{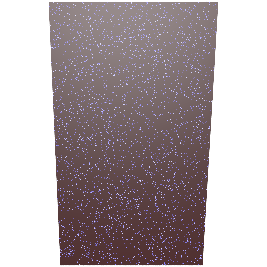}
\caption{Rectangle}
\label{subfig:S1init}
\end{subfigure}
\begin{subfigure}{0.15\linewidth}
\includegraphics[width = \linewidth]{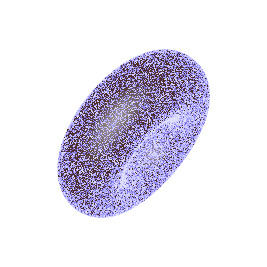}
\caption{Torus}
\label{subfig:S2init}
\end{subfigure}
\begin{subfigure}{0.15\linewidth}
\includegraphics[width = \linewidth]{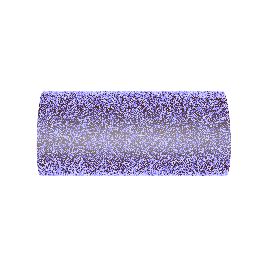}
\caption{Cylinder}
\label{subfig:S3init}
\end{subfigure}
\caption{Initial surfaces with uniformly distributed random sample points.}
\label{fig:InitSurfs}
\end{figure}

\begin{figure}[h]
\centering
\includegraphics[width = 0.85\linewidth]{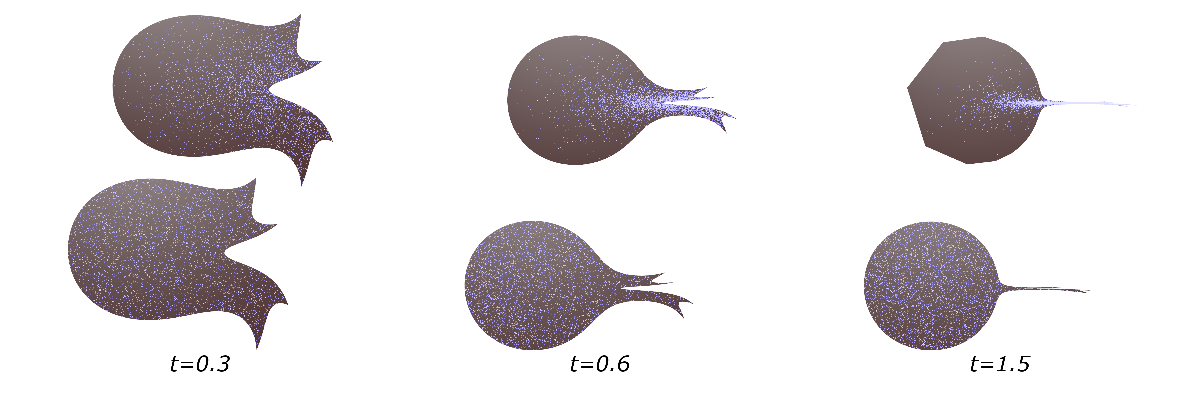} \\ 
\caption{Comparison of the surface sampling of the evolved rectangle using the original parametrization $\vP$ (top images) and the redistributed parametrization $\vQ$ (bottom images).}
\label{fig:S1Evolve}
\end{figure}

\begin{figure}[h]
\centering
\includegraphics[width = 0.85\linewidth]{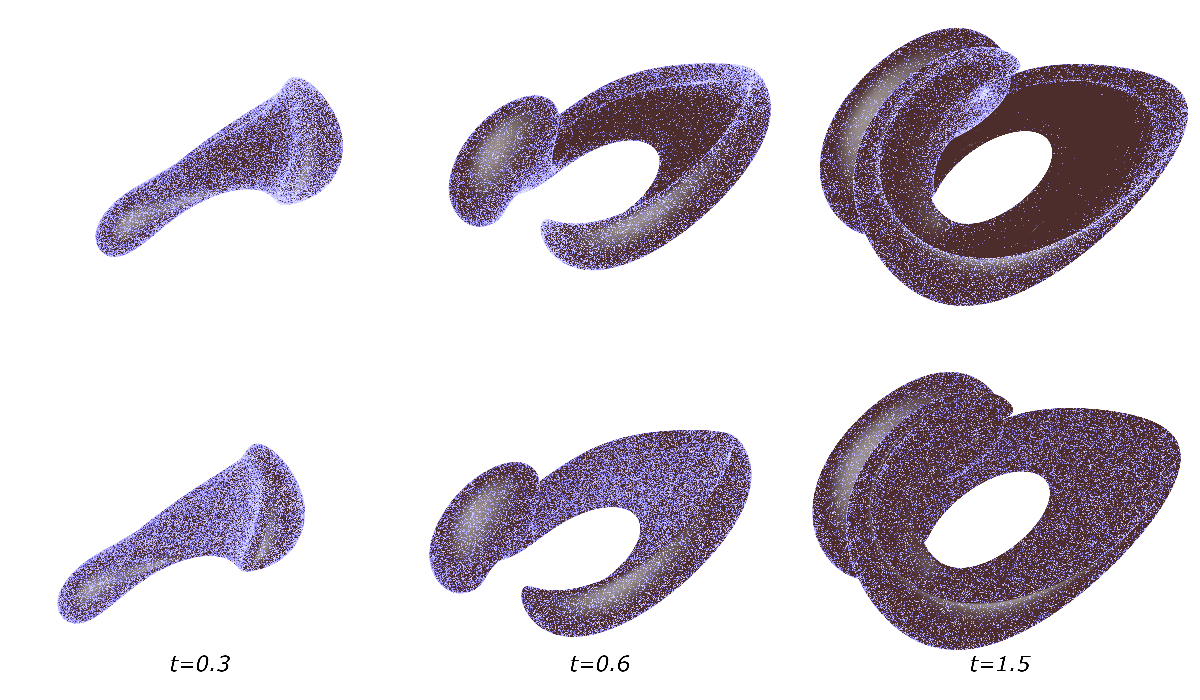} \\ 
\caption{Comparison of the surface sampling of the evolved torus using the original parametrization $\vP$ (top images) and the redistributed parametrization $\vQ$ (bottom images).}
\label{fig:S2Evolve}
\end{figure}

\begin{figure}[h]
\centering
\includegraphics[width = 0.85\linewidth]{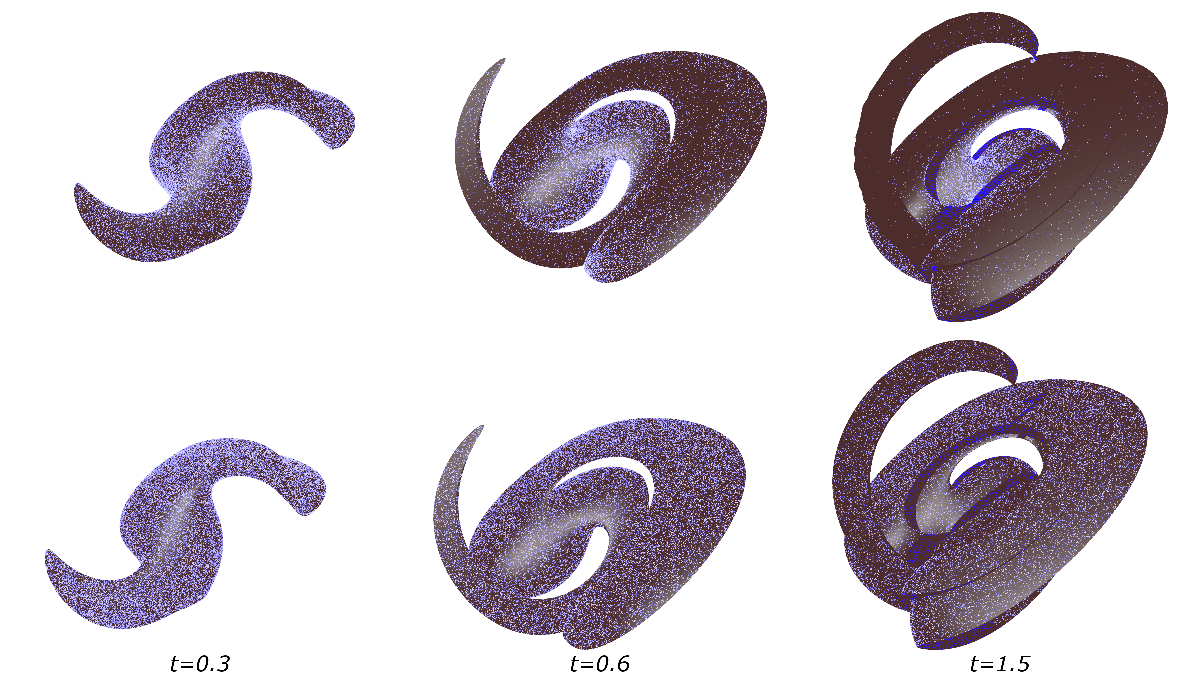} \\ 
\caption{Comparison of the surface sampling of the evolved cylinder using the original parametrization $\vP$ (top images) and the redistributed parametrization $\vQ$ (bottom images).}
\label{fig:S3Evolve}
\end{figure}

Figure \ref{fig:S1Evolve} shows a clear example of the benefits of the reparametrization method. In this case, the initial rectangle is placed on one of the symmetry planes of the flow, hence the deformation is applied fully on the tangential directions of the surface. At $t = 1.5$, we see from the $\vP$ parametrization that almost the entirety of the surface is compressed in a thin sliver. The sample points distribution is unnecessarily dense on the thin protrusion and very sparse on the rest of the surface. In fact, using a $512^2$ grid for the interpolation of $\vP$, we still see that the boundary of the surface is jagged and visibly piecewise-linear as opposed to the smooth circular shape shown by the same resolution interpolation of $\vQ$. Indeed, each line segment on the boundary corresponds to a cell edge of length $\bigO (1/512)$ at time $0$. This is an indication that the large scale deformation the surface has undergone between $t=0$ and $1.5$ makes the $\vP$ parametrization inefficient; the numerical difficulties from this deformation are mitigated by the redistribution method resulting in a reparametrization $\vQ$ which provides an uniform sampling and a smooth, well-resolved surface interpolation. Similar observations can be made in figure \ref{fig:S2Evolve} where, at $t=1.5$, a region of the torus is essentially not sampled under the $\vP$ parametrization. The clear demarcation line between the sampled and empty regions is in fact produced by the perspective of the view angle on the hole of the torus after the flow deformation. Hence, the $\vP$ sampling indicates that the marker points failed to represent a handle of the genus-1 surface; if a ``pure'' particles method were used, this can cause errors in the topology of the evolved shapes. For the evolution of the cylinder in figure \ref{fig:S3Evolve}, we also see the above issues in the $\vP$ parametrization. At $t=1.5$, the ``top face'' of the surface, which consists of two diametrically opposite portions of the cylinder that were brought close together by the flow, is poorly sampled by $\vP$. Without some underlying interpolation of the parametrization, such undersampling could fail to indicate the presence of two distinct pieces of surface. Furthermore, the protruding arc-like part also exhibit poor resolution of the boundary. The piecewise linear interpolation of $\vP$ is jagged at the boundary, which indicates that $\vP$ is not smooth enough (in the sense of the growth in magnitude of the higher derivatives) to be accurately represented on a $512^2$ grid. With the redistributed parametrization, the marker points generated from $\vQ$ are uniform and the interpolated surface, smooth.

We quantify the effect of the redistribution by plotting the histogram of the cell area population in figure \ref{fig:areaDist}. In all cases, we see that the $\vQ$ parametrization (in red/orange) generates cells that have almost all the same area, concentrated at the normalized average 1. The $\vP$ parametrization (in blue) on the other hand, generates large disparity between cell areas, evidential of a non-uniform, suboptimal distribution of marker points. In particular, for the rectangle, we see in figure \ref{subfig:S1_areadist}, that the area distribution exhibits a Pareto distribution, where the large amount of the surface area is represented by a minority of the cells; this observation clearly reflects the illustrations in figure \ref{fig:S1Evolve}. The reparametrization $\vQ$ is more uniform, with almost all cells having the average area. These properties are also shown quantitatively in table \ref{tab:Sstats} where we've computed the sample standard deviation and median of the cell areas. In all cases, the standard deviation from the $\vQ$ parametrization is 1 to 2 orders of magnitude smaller than the one from $\vP$, and the median error about 3 to 4 orders of magnitude smaller. In all cases, these improvements came at a cost of a roughly $20\%$ increase over the computational time of the advection (given in section \ref{subsec:equidistCS}); this reparametrization time includes the intermittent evaluation of the full $\vQ$ parametrization (as defined in \eqref{eq:movedParamNum}) which is also used to render the surface. The computation of the redistribution map itself accounts for about $10\%$ of the total computation time. The resulting $\vQ$ however, is a functionally defined parametrization which offers arbitrary resolution at uniform area density.

\begin{figure}[h]
\centering
\begin{subfigure}{0.32\linewidth}
\includegraphics[width = \linewidth]{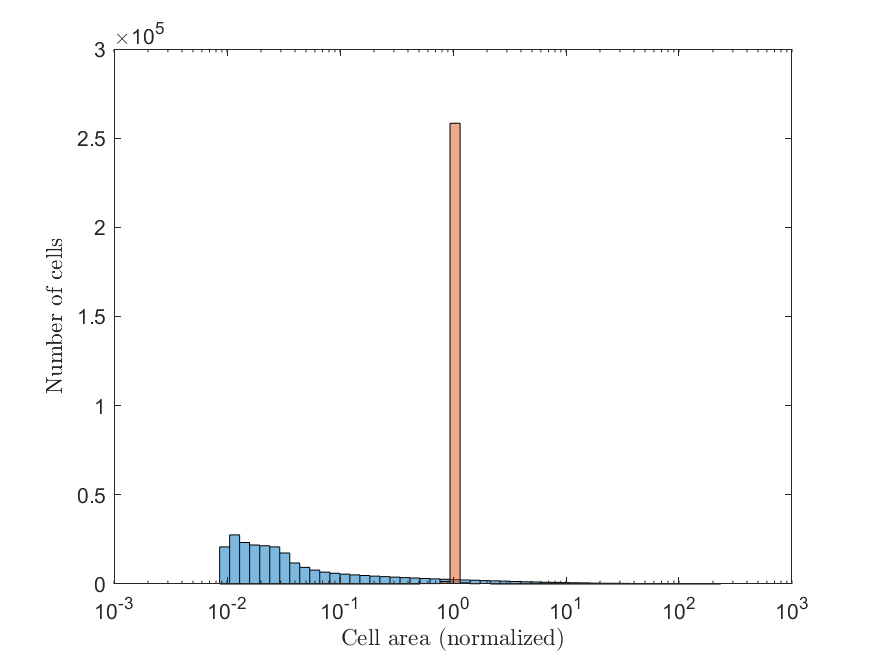}
\caption{Rectangle}
\label{subfig:S1_areadist}
\end{subfigure}
\begin{subfigure}{0.32\linewidth}
\includegraphics[width = \linewidth]{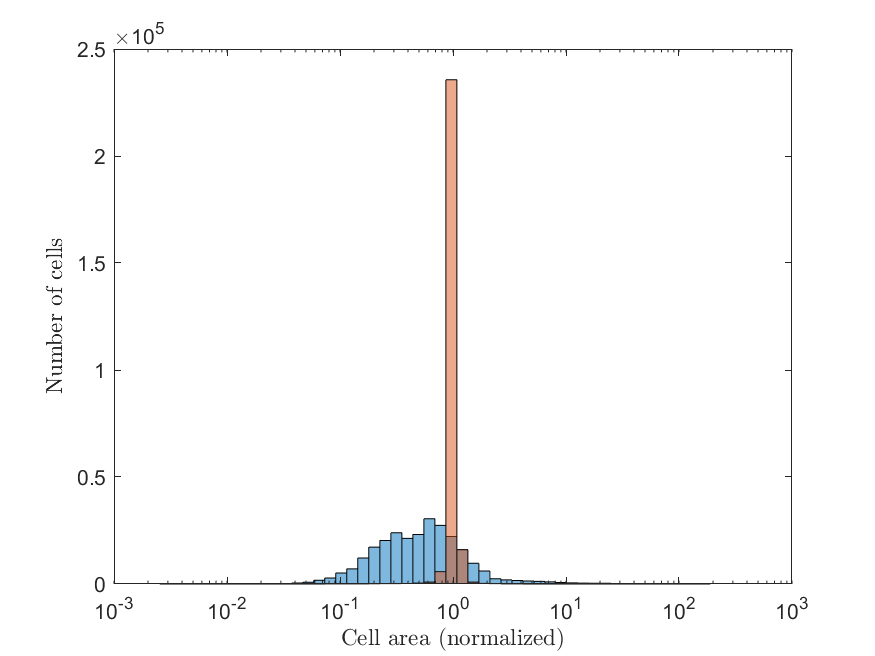}
\caption{Torus}
\label{subfig:S2_areadist}
\end{subfigure}
\begin{subfigure}{0.32\linewidth}
\includegraphics[width = \linewidth]{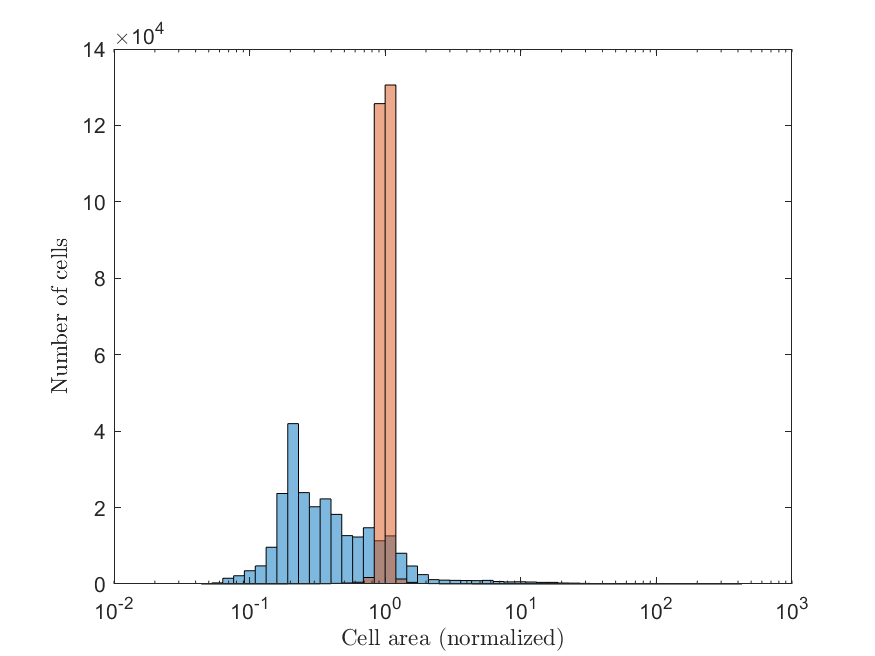}
\caption{Cylinder}
\label{subfig:S3_areadist}
\end{subfigure}
\caption{Area distribution of the parametrizations $\vP$ (blue) and $\vQ$ (red)}
\label{fig:areaDist}
\end{figure}

\begin{table}[H] \footnotesize
\begin{subtable}{0.4\linewidth}
\begin{center}
%{\renewcommand{\arraystretch}{1.5}
\begin{tabular}{c | c c c}
\hline
\hline
Surfaces & Rect. & Torus & Cyl. \\
\hline
$\sigma_{\vP}$ & 1.6749 & 0.8828 & 1.098 \\
$\sigma_{\vQ}$  & 0.0094 & 0.0485 & 0.0321 \\
\hline
$M_{\vP}-1$  & -0.9722 & -0.4754 & -0.6718 \\
$M_{\vQ}-1$  & -0.0007 & -0.0001 & -0.0019 \\
\hline
\end{tabular}%} \\
\end{center}
\captionsetup{width=.95\linewidth}
\caption{Standard deviation ($\sigma$) and median ($M$) errors of the length density.}
\label{subtab:Serrors}
\end{subtable}
\begin{subtable}{0.5\linewidth}
\begin{center}
%{\renewcommand{\arraystretch}{1.5}
\begin{tabular}{c | c c c}
\hline
\hline
Surfaces & Rect. & Torus & Cyl.  \\
\hline
Evaluating $\vQ$ & 47.94 s & 64.68 s & 54.57 s \\
Defining $\rho^n(\vx)$ & 9.15 s & 9.36 s & 9.25 s \\
Updating $\vhX_B$ & 68.86 s & 44.29 s & 49.06 s \\
Number of remappings & 5 & 11 & 5 \\
\hline
\end{tabular}%} \\
\end{center}
\captionsetup{width=.95\linewidth}
\caption{Total computation times for the evolution of the parametrization $\vQ$.}
\label{subtab:Stimes}
\end{subtable}
\caption{Parametrizations $\vP$ and $\vQ$ at $t=1.5$.}
\label{tab:Sstats}
\end{table}

\begin{figure}[h]
\centering
\begin{subfigure}{0.28\linewidth}
\includegraphics[width = \linewidth]{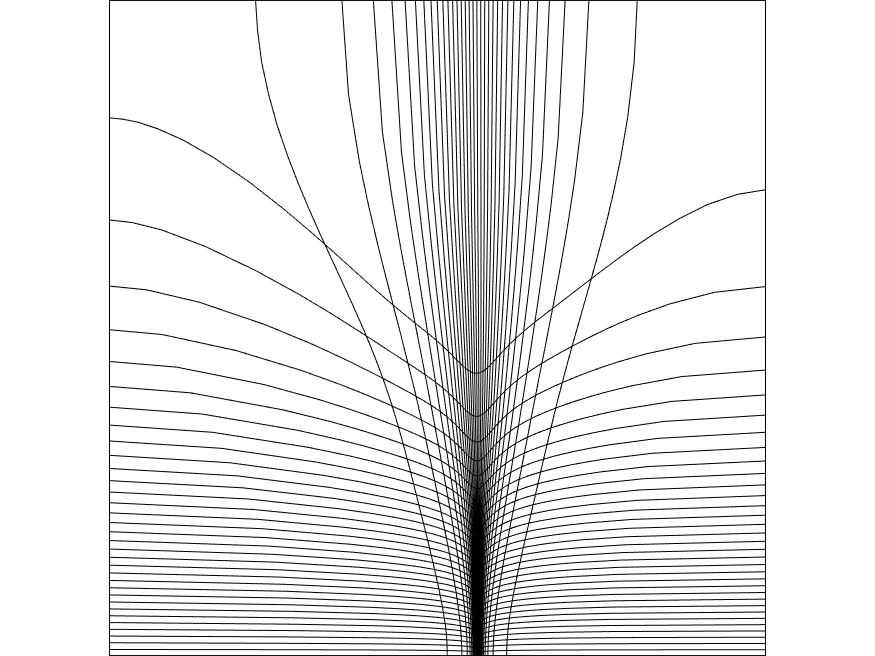}
\caption{Rectangle}
\label{subfig:S1map}
\end{subfigure}
\begin{subfigure}{0.28\linewidth}
\includegraphics[width = \linewidth]{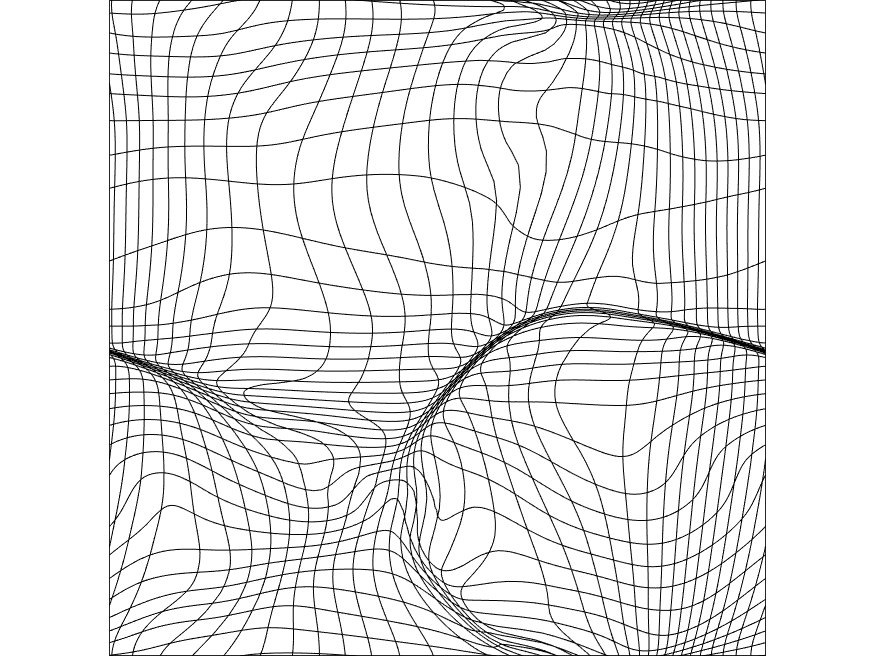}
\caption{Torus}
\label{subfig:S2map}
\end{subfigure}
\begin{subfigure}{0.28\linewidth}
\includegraphics[width = \linewidth]{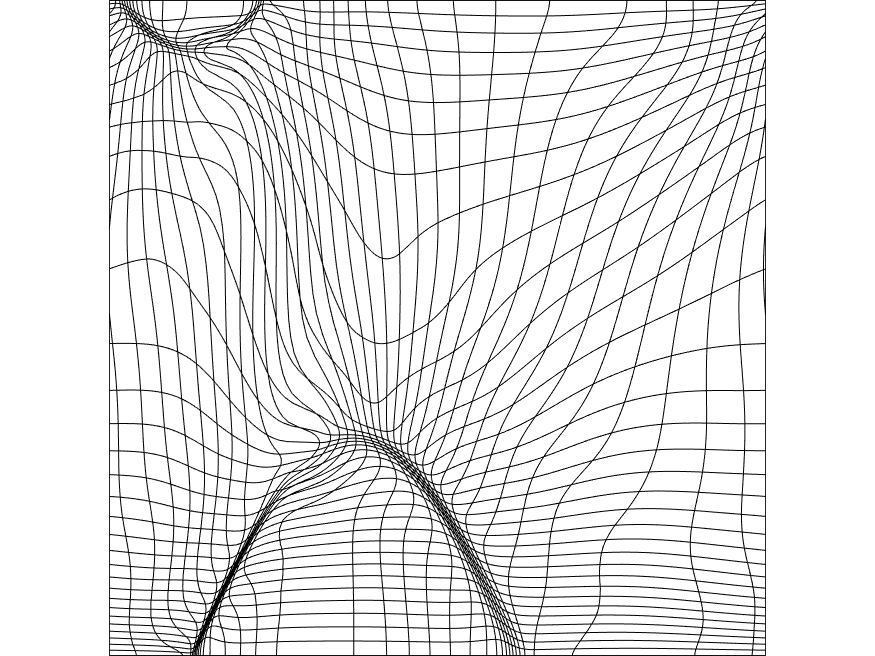}
\caption{Cylinder}
\label{subfig:S3map}
\end{subfigure}
\caption{The redistribution map on the parametric space for each surface at $t=1.5$.}
\label{fig:SMaps}
\end{figure}
Indeed, figure \ref{fig:SMaps} shows the redistribution map on the parametric space for each surface at full deformation. As we can see, since the surfaces undergo violent deformations, the transport maps needed to redistribute evenly the area element will also contain large deformations and small scale features: typically this would require a fine grid to compute and resolve. Instead, we use the semigroup structure of the characteristic maps to carry out short-time computations on coarse grids, the full time redistribution map is then obtained from the composition of submaps given in \eqref{eq:movedParamNum}. The advantage of using the decomposition method for the redistribution map is two-fold: first, if we assume that at time $\tau_i$, the redistributed parametrization $\vQ_{\tau_i}$ is close to equiareal, then the variations in $\rho(\vx, \tau_i)$ are small and therefore can be resolved on a coarse grid for the computation of $\vhX_{[\tau_{i+1}, \tau_i]}$. The resolution of this grid essentially acts as a frequency cut-off selecting the highest frequency in the area element visible to the algorithm. Second, since the evolution of the area density is unknown \emph{a priori} and can grow in an arbitrary fashion, the semigroup property of the map allows us to decompose the transformation into manageable short-time maps and achieve arbitrary resolution in the global-time map obtained from the composition. This permits the redistribution map to represent large deformations and resolve small scale features in order to compensate for the arbitrary changes in area element.

%\FloatBarrier

%\bigskip
The numerical experiments in this section demonstrate several practical advantages of representing a moving surface using an equiareal parametrization. Compared to a particle-based method, the parametrization of a moving surface defined by the push-forward of the initial parametrization by the flow map provides a functional definition of the surface at all times. Therefore, resampling can be done by simply evaluating the parametrization function at new sample points. In particle methods, new sample points need to be generated by interpolation which can affect the position of the surface whose accuracy will depend on existing sample points and the curvature of the surface. With the CM method, the accuracy of the parametrization function in respect to the surface shape and location depends only on the forward characteristic maps which we control separately. The precision of the marker locations is therefore independent of existing sample points and of the shape of the surface. Furthermore, coupled with the redistribution method, the equiareal property of the parametrization function is maintained. This means that the sampling density of the parameter space is directly mapped to that of the surface without needing extra computation. Adaptive sampling methods based on the area element of the surface may become inaccurate or inefficient when the variations in area become large. In comparison, the method proposed in this paper maintains an equiareal parametrization at all times, therefore there are no extra computations necessary when computing an uniform sampling of the surface and area features above a certain spatial scale as indicated in theorem \ref{thm:areaErrorBound} are guaranteed to be resolved.

\begin{remark}
The redistribution generates an equiareal parametrization, from a sampling point of view, the number of sample points per unit area on the surface should be asymptotically constant. This does not directly translate to a property on the distance between sample points. Indeed, the ratio between the geodesic radius of a disk and its area on the surface is given by the scalar curvature. As a consequence, an equiareal sampling of a surface will have sample points that are more distant from each other where the surface has positive scalar curvature. One way to remedy this would be to require that the equilibrium density of the diffusion be given by the curvature of the surface:
\begin{gather}
\partial_t \rho = \Laplace (\rho - \rho_\kappa) \quad \text{with} \quad \vu = - \frac{\grad (\rho - \rho_\kappa)}{\rho},
\end{gather}
where $\rho_\kappa$ is the target density given by the curvature. The redistribution will then generate an $L^2$-gradient descent on the difference between $\rho$ and $\rho_\kappa$. However, since we no longer have the maximum principle due to the source term, the solution of the heat equation is not guaranteed to stay in the space of probability distributions for all times as densities can temporarily become negative, potentially making the redistribution map singular. So special consideration needs to be taken when the initial and equilibrium densities are far since we no longer have the maximum principle due to the source term. However, this should generally not be an issue if the curvature changes gradually as the source term will be too small to generate a singularity.
\end{remark}

\section{Conclusion}
In this work, we have presented a novel method for computing an equiareal parametrization of a curve or surface. We have applied this method to the problem of surface advection and presented an algorithm for evolving the parametrization of a time-dependent moving surface while maintaining an uniform area distribution. The area element of a surface is first pulled back to the parametric space in order to define a density distribution. We define a diffusion process using this density as initial condition and extract a velocity field from the continuity equation, flowing the initial density towards the uniform density. The backward characteristic map generated from this velocity is used to equidistribute the area element of the surface in the parametric space. In the context of optimal transport, this can be seen as a $W^2$-gradient descent of the entropy landscape. We studied the mathematical construction of the map in section \ref{subsec:densityRedist} and demonstrated the convergence in distribution of the random-variables generated from it. The evolution of the surface under a velocity field is obtained by computing the forward characteristic map in the ambient space as studied in section \ref{sec:surfAdv}. Combined with the redistribution method, we have that the composition of the backward redistribution map on the parametric space with the initial parametrization followed by the push-forward by the ambient space forward map generates an equiareal parametrization for all times. We then tested this method and provided numerical examples of evolving surfaces in 3D ambient space in section \ref{sec:numResults}. This method is novel and unique in that the changes in the parametrization function are made by pre-composition with a deformation of the parametric domain. As a consequence, the image space of the parametrization function is unaffected and we preserve the correct position and shape of the surface. Furthermore, the characteristic mapping method allows us to exploit the semigroup structures of both the surface advection in the ambient space and the density transport on the parametric space. This allows for the decomposition of both maps into coarse grid submaps of smaller deformations, while maintaining high resolution for the parametrization obtained from the composition. The resulting method is able to track large deformations of the curves and surfaces in the ambient space and redistribute the resulting large variations in area density on the parametric space. 

The use of the CM method for evolving equiareal parametrizations of surfaces opens a novel framework with many possibilities for future research. Although it is generally impossible to generate a parametrization which is both equiareal and conformal, there is possibility to maintain a parametrization with a trade-off between these two properties using the CM framework by devising an appropriate redistribution velocity. There may also be interesting methods that couple the equiareal parametrization evolution with a meshing algorithm to generate high quality triangulation meshes on moving surfaces, application of such methods to interface problems such as the Cahn-Hilliard equations \cite{gera2017cahn} could also be of interest. Additionally, since the redistribution framework generalizes directly to any number of dimensions, it may be interesting to investigate its application to volume redistributing flows in 3-dimensional space. Lastly, from a geometric point of view, it would be interesting to study the relation between the diffusion-driven redistribution flow and the geometric flows in the theory of manifold uniformization. We think that these directions of research could be interesting as we believe that the CM method provides a novel and unique approach for solving the problem of surface parametrization and sampling.

\section*{Acknowledgements}
The work by X.-Y. Y. was partially supported by FRQNT B2X (Fonds de recherche du Qu\'{e}bec – Nature et technologies) and by Hydro-Qu\'{e}bec. The work by L. C. was partially supported by the NSERC Discovery program (Natural Sciences and Engineering Research Council). The work by J.-C. N. was partially supported by the NSERC Discovery program.

\clearpage
\appendix
\section{Time Evolution of 1D Curves} \label{sec:appendixCurves}
The curves used in section \ref{subsec:evoCurves} are shown below. The initial curves are shown in figure \ref{fig:InitCurves}, the curves at times $t = 0.6$, $0.9$ and $1.5$ are shown in figure \ref{fig:Curves}. We sample each parametrization function with 256 random marker points sampled from uniform distribution on $U$ which we draw as blue dots over the underlying exact curve in black. The sampling using the $\vP$ parametrization are shown in the top images, the reditributed $\vQ$ parametrization, in the bottom images.
%\clearpage
\begin{figure}[h]
\centering
\begin{subfigure}{0.15\linewidth}
\includegraphics[width = \linewidth]{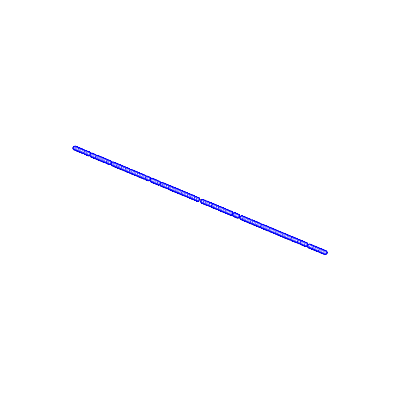}
\caption{Curve 1}
\label{subfig:C1init}
\end{subfigure}
\begin{subfigure}{0.15\linewidth}
\includegraphics[width = \linewidth]{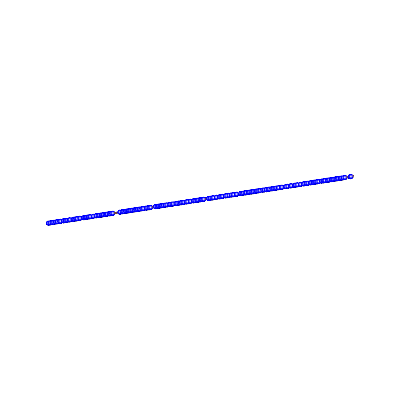}
\caption{Curve 2}
\label{subfig:C2init}
\end{subfigure}
\begin{subfigure}{0.15\linewidth}
\includegraphics[width = \linewidth]{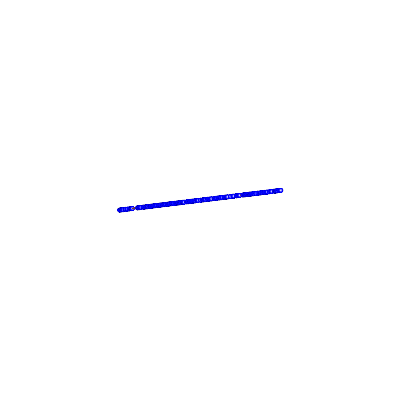}
\caption{Curve 3}
\label{subfig:C3init}
\end{subfigure}
\begin{subfigure}{0.15\linewidth}
\includegraphics[width = \linewidth]{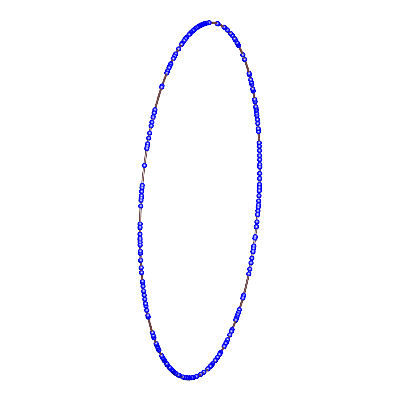}
\caption{Curve 4}
\label{subfig:C4init}
\end{subfigure}
\caption{Initial curves with random sampling.}
\label{fig:InitCurves}
\end{figure}
%\begin{sidewaysfigure}
\begin{figure}[h]
\centering
\begin{subfigure}{0.45\linewidth}
\includegraphics[width = \linewidth]{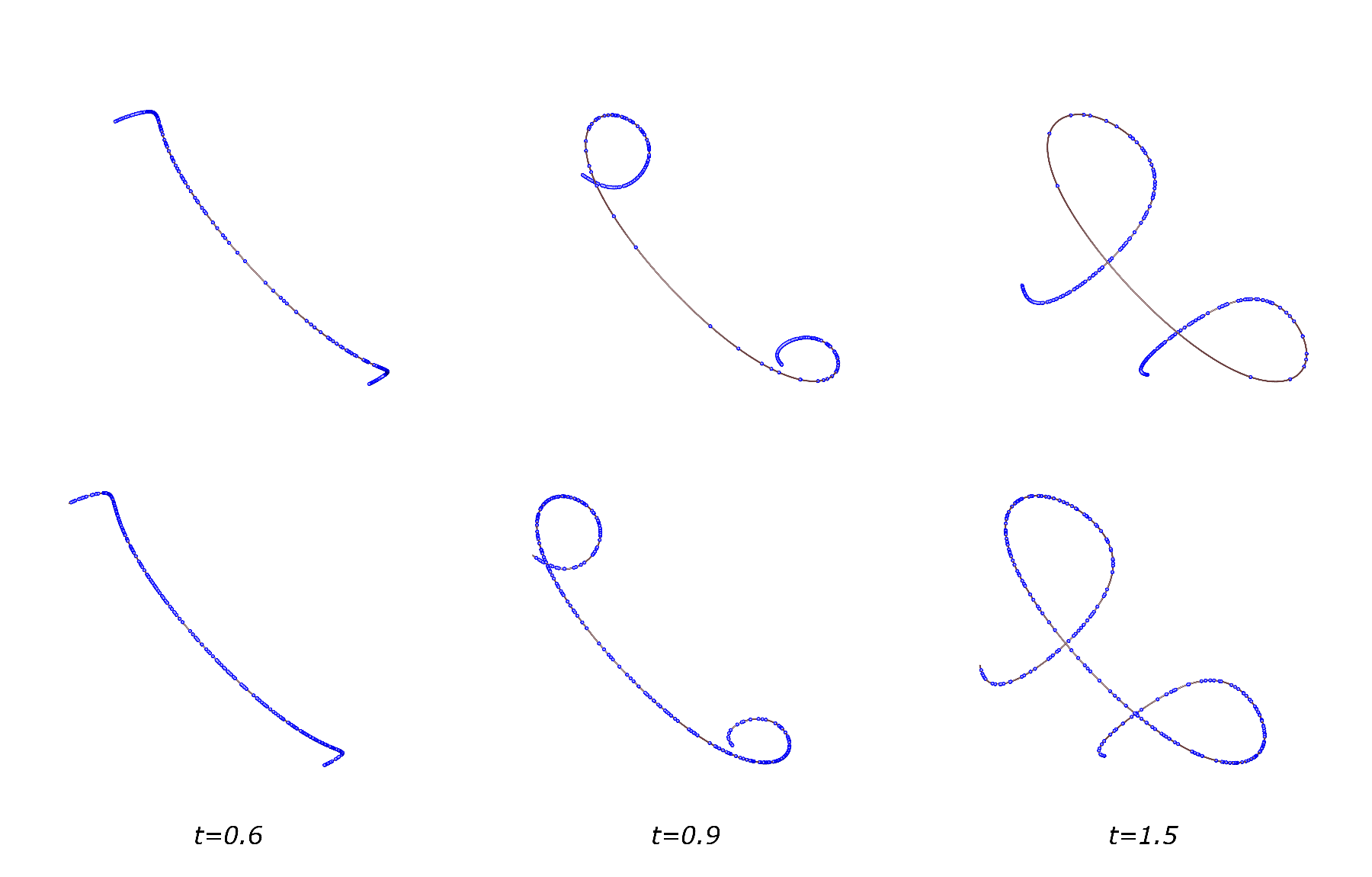}
\caption{Curve 1}
\label{subfig:C1}
\end{subfigure}
\begin{subfigure}{0.45\linewidth}
\includegraphics[width = \linewidth]{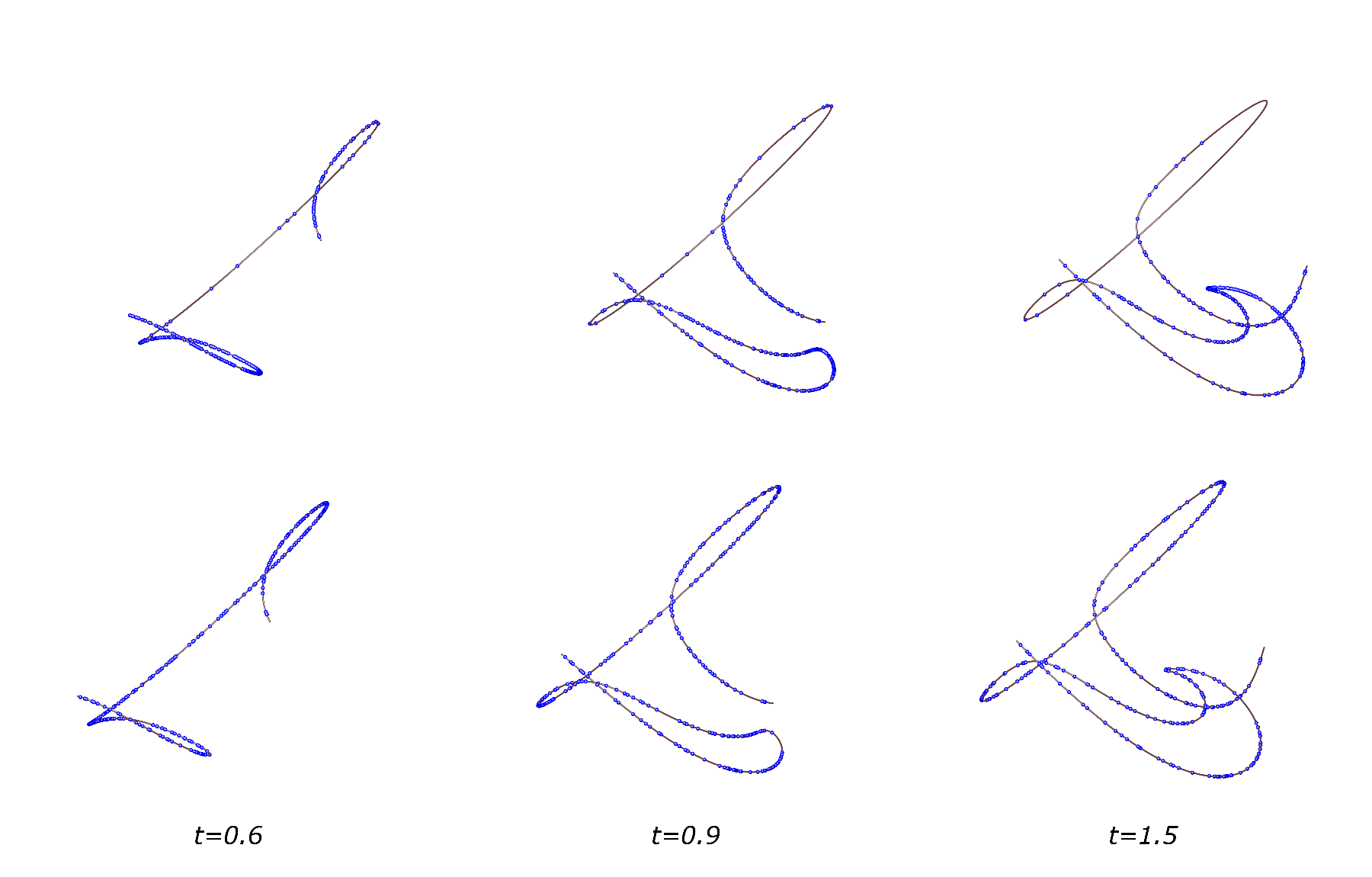}
\caption{Curve 2}
\label{subfig:C2}
\end{subfigure}
\begin{subfigure}{0.45\linewidth}
\includegraphics[width = \linewidth]{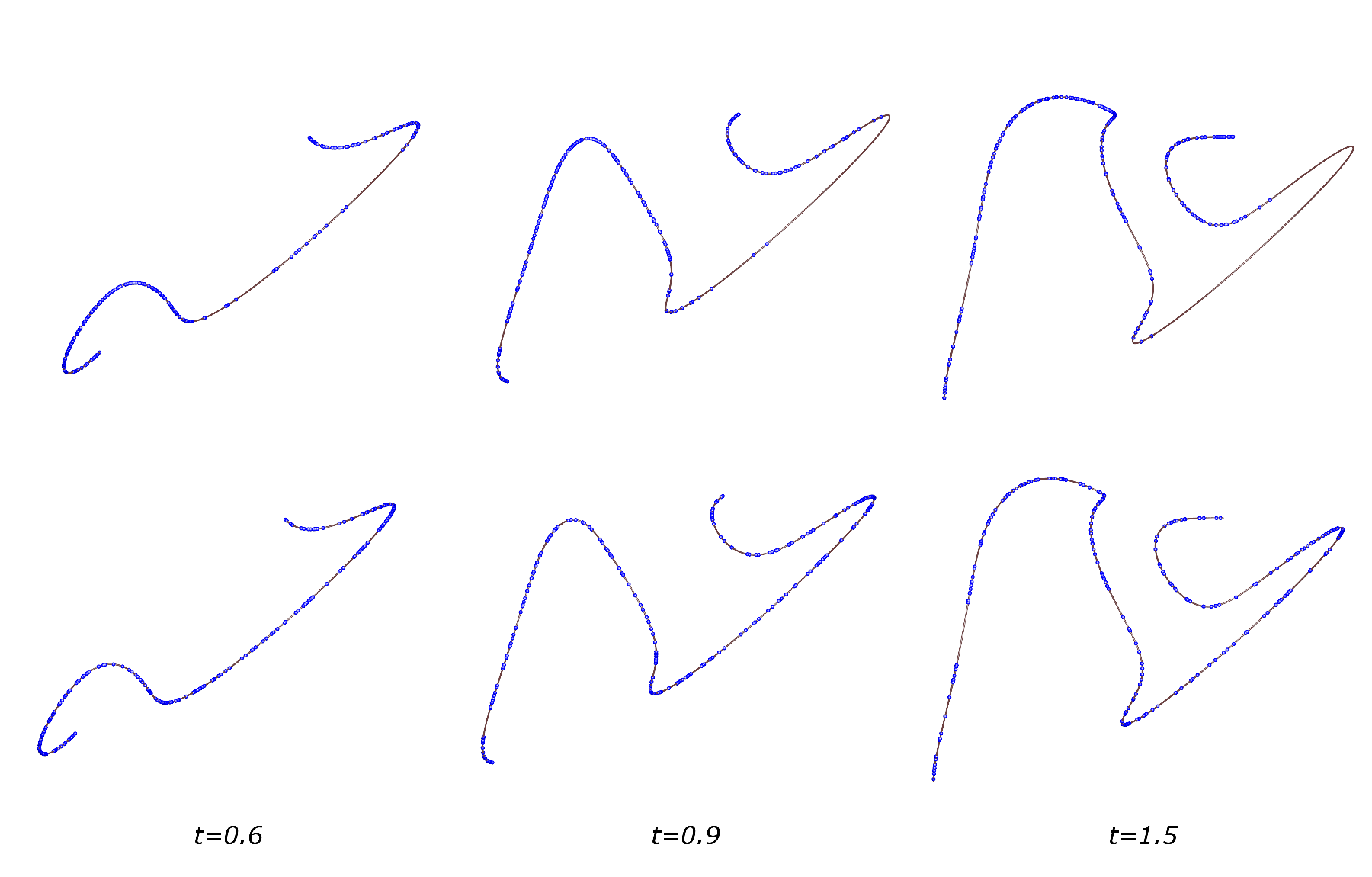}
\caption{Curve 3}
\label{subfig:C3}
\end{subfigure}
\begin{subfigure}{0.45\linewidth}
\includegraphics[width = \linewidth]{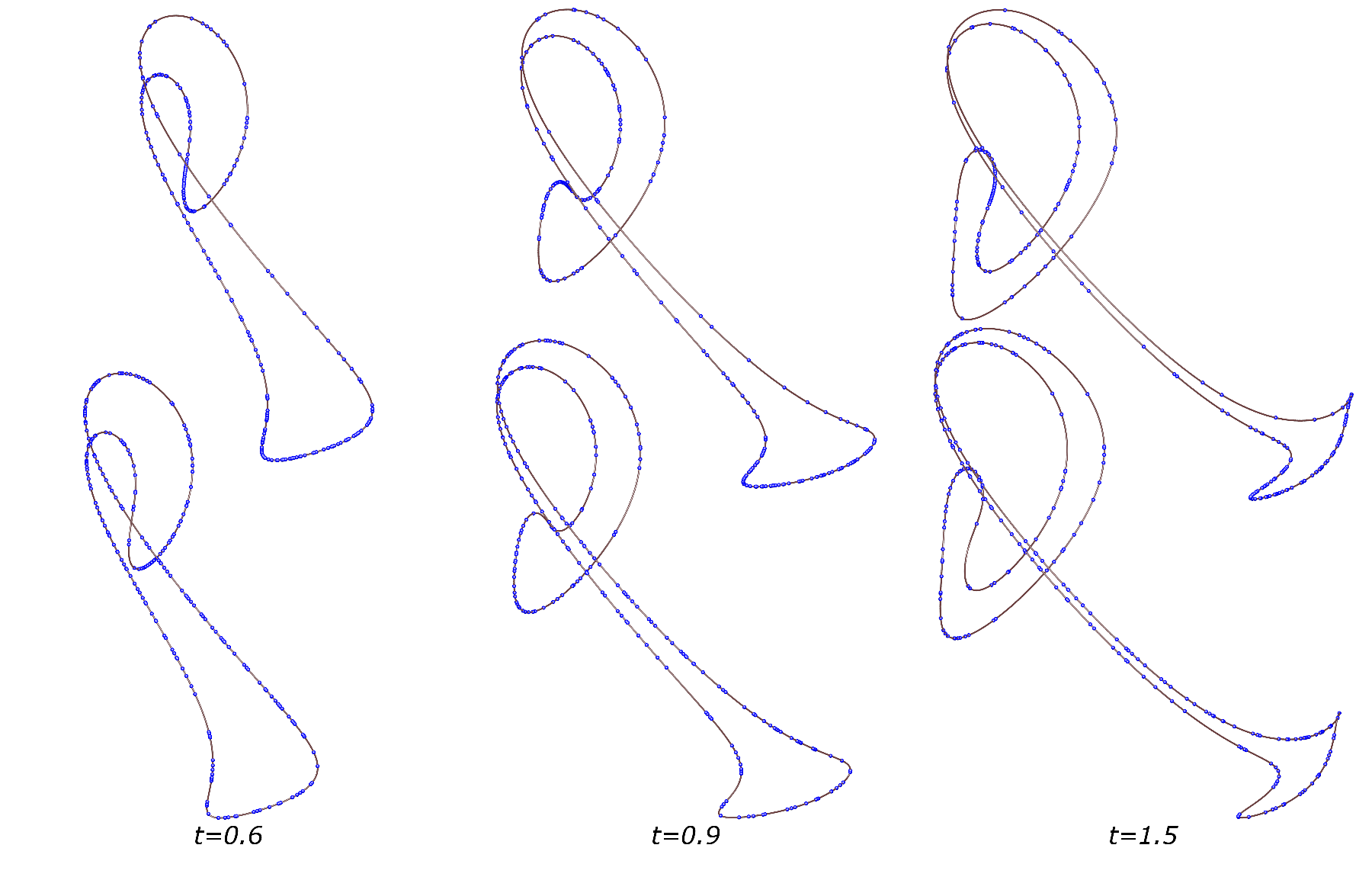}
\caption{Curve 4}
\label{subfig:C4}
\end{subfigure}
\caption{Comparison of the curves sampling using the original parametrization $\vP$ (top images) and the redistributed parametrization $\vQ$ (bottom images).}
\label{fig:Curves}
%\end{sidewaysfigure}
\end{figure}

\clearpage
%\FloatBarrier
\vspace*{12pt} %trying to push bib entry 29 to next page to avoid cut by pagebreak
\bibliography{densityTransport}
\bibliographystyle{siamplain}

%%\end{thebibliography}
\end{document}